\documentclass[reqnot,12pt]{amsart}

\usepackage{amsmath, amsfonts, amssymb, hyperref}
\usepackage{graphics}
\usepackage[usenames]{color}

{\theoremstyle{plain}
\newtheorem{theorem}{Theorem}[section]

\newtheorem{proposition}[theorem]{Proposition}
\newtheorem{lemma}[theorem]{Lemma}
\newtheorem{corollary}[theorem]{Corollary}
}
{\theoremstyle{definition}
\newtheorem{definition}[theorem]{Definition}
\newtheorem*{definition*}{Definition}
\newtheorem*{proposition*}{Proposition}
\newtheorem{remark}[theorem]{Remark}
\newtheorem{example}[theorem]{Example}

}

\numberwithin{equation}{section}

\newcommand{\kroncoeff}{C}  % my feeling is that we should not use a lower case greek
\newcommand{\charf}[1]{\big(\!\!\big( #1 \big)\!\!\big)}
\newcommand{\rug}[1]{[ #1 ]}

\newcommand{\comment}[1]{\vspace{5 mm}\par \noindent
\marginpar{\textsc{Comment}}
\framebox{\begin{minipage}[c]{0.95 \textwidth}
 #1 \end{minipage}}\vspace{5 mm}\\}

\renewcommand{\comment}[1]{}

\begin{document}
\title[Catalan Kroneckcer products]{Expressions for Catalan Kronecker products}

\author{Andrew A.H. Brown}
\address{Department of Mathematics, University of British Columbia, Vancouver, BC V6T 1Z2, Canada}
\curraddr{Department of Combinatorics and Optimization,
University of Waterloo,
Waterloo, ON  N2L 3G1, Canada}
\email{}

\author{Stephanie van Willigenburg}
\address{Department of Mathematics, University of British Columbia, Vancouver, BC V6T 1Z2, Canada}
\email{steph@math.ubc.ca}

\author{Mike Zabrocki}
\address{Department of Mathematics and Statistics, York University, Toronto, ON M3J 1P3, Canada}
\email{zabrocki@mathstat.yorku.ca}
\thanks{The  authors were supported in part by the National Sciences and Engineering Research Council of Canada.}

\begin{abstract} 
We give some elementary manifestly positive formulae for the Kronecker 
products $s_{(d,d)} \ast s_{(d+k,d-k)}$ and $s_{(d,d)} \ast s_{(2d-k,1^k)}$. 
These formulae demonstrate some fundamental
properties of the Kronecker coefficients, and we use them to deduce a number
of enumerative and combinatorial results.
%including what the maximum coefficient is in a given product and where it lies. 
\end{abstract}

\maketitle

\section{Introduction}\label{sec:intro} 
A classic open problem in  algebraic combinatorics is
to explain the Kronecker product (or internal product) of two Schur functions
in a manifestly positive combinatorial formula.
The Kronecker product  of two Schur functions is the
Frobenius image of the internal tensor product of two irreducible symmetric group modules, or
it is alternatively the characters of the 
induced tensor product of general linear group modules.
Although this expression clearly has non-negative coefficients when
expanded in terms of Schur functions for representation theoretic reasons,
it remains an open problem to provide a satisfying positive
combinatorial or algebraic formula for the Kronecker 
product of two Schur functions.

Many attempts have been made to capture some aspect of these coefficients, for example,  special cases \cite{BKron1, BKron2,  RemmelWhitehead, Rosas}, asymptotics \cite{BOrellana1, BOrellana2}, 
stability \cite{Vallejo},
the complexity of calculating them \cite{FPSAC08}, and conditions when they are
non-zero \cite{Dvir}.  Given that the Littlewood-Richardson rule and many successors 
have so compactly and cleanly been able to describe the
external product of two Schur functions it seems as though some new ideas
for capturing the combinatorics of Kronecker coefficients are needed.

The results in this paper were inspired by the symmetric function identity  for the Kronecker product of two Schur functions, $s_{(d,d)} \ast s_{(d,d)}$ that appears as \cite[Theorem I.1]{GWXZ1}.  More precisely, for a subset of partitions $X$ of $2d$, if we set $\rug{X} = \sum_{\lambda \in X} s_\lambda$, called a \emph{rug},
then 
\begin{equation}\label{sddsdd}
s_{(d,d)} \ast s_{(d,d)} = \rug{\hbox{ $4$ parts all even or all odd }}~.
\end{equation} 
This particular identity is significantly different
from most published results on the Kronecker product because it clearly states exactly which
partitions have non-zero coefficients and that all of the coefficients
are 0 or 1 instead of giving a combinatorial interpretation or algorithm.

This computation originally arose in
the solution to a mathematical physics problem related
to resolving the interference of 4 qubits \cite{Wallach} because the sum of these coefficients
is equal to the dimensions of polynomial invariants of four copies of 
$SL(2, {\mathbb C})$ acting on ${\mathbb C}^{8}$. Understanding the Kronecker
product of $s_{(d,d)}$ with $s_\lambda$ for partitions $\lambda$ with $4$ parts
which are all even or all odd would be useful for 
calculating the dimensions of invariants of six copies of $SL(2, {\mathbb C})$
acting on ${\mathbb C}^{12}$ which is a measure of entanglement of 6 qubits.
Ultimately we would like to be able to compute
$$CT_{a_1, a_2, \ldots, a_k} \left(\frac{\prod_{i=1}^k (1-a_i^2)}{\prod_{S \subseteq \{1,2,\ldots k\}} (1-q\prod_{i \in S} a_i/ \prod_{j \notin S} a_j)} \right) = \sum_{d \geq 0}
\left< s_{(d,d)}^{\ast k}, s_{(2d)} \right> q^{2d}$$
(see \cite[formulas I.4 and I.5]{GWXZ2} and \cite{LT} for a discussion) where $CT$ represents the
operation of taking the constant term and equations of this type is a motivation for
understanding the Kronecker product with $s_{(d,d)}$ as completely as possible.

Using  $s_{(d,d)} \ast s_{(d,d)}$ as our inspiration, we were able to show (Corollary~\ref{cor:easycases}) that
$$
s_{(d,d)} \ast s_{(d+1,d-1)} = \rug{\hbox{ $2$ even parts and $2$ odd parts } }$$
and with a similar computation we were also able to derive that
\begin{equation*}
s_{(d,d)} \ast s_{(d+2,d-2)} = \rug{ \hbox{ 4 parts, all even or all odd, but not 3 the same } } + 
\rug{ \hbox{ 4 distinct parts } }~.
\end{equation*}
%where we have the sum over partitions $\lambda$ is over two 
%(intersecting for $d\geq 6$)
%sets of partitions $A$ and $B$.  Specifically $A$ is the set of partitions of $2d$ with $4$ even parts or
%$4$ odd parts but does not have three equal parts and $B$
%is the set of partitions of $2d$ with $4$ distinct parts.  
One interesting feature of this formula is that it says that
all of the coefficients in the Schur expansion of $s_{(d,d)} \ast s_{(d+2,d-2)}$ are either
$0,1$ or $2$ and the coefficient is $2$ with those Schur functions indexed by partitions
with 4 distinct parts that are all even or all odd.  

These and larger examples suggested that the Schur function expansion of 
$s_{(d,d)} \ast s_{(d+k,d-k)}$ has the pattern of a boolean lattice of subsets,
in that it can be written as the sum of $\lfloor k/2 \rfloor +1$ intersecting
sums of Schur functions each with coefficient $1$.  The main result of this article is
Theorem \ref{cleanest}, which states
\begin{equation}
s_{(d,d)} \ast s_{(d+k,d-k)} = \sum_{i=0}^{k} \rug{ (k+i,k,i)P } + \sum_{i=1}^k  
\rug{ (k+i+1,k+1,i)P }
\end{equation}
where we have used the notation $\gamma P$ to represent the set of partitions $\lambda$ of
$2d$ of length less than or equal to $4$ such that $\lambda - \gamma$ (representing a
vector difference)
is a partition with 4 even parts or 4 odd parts.
%
%
%
%\begin{center}
%{\setlength{\unitlength}{0.00083333in}
%\begin{picture}(1741,1081)(0,-10)
%\put(533,533){\circle{1050}}
%\put(900,533){\circle{1050}}
%\put(250,533){A}
%\put(1000,533){B}
%\end{picture}
%}\end{center}
%In the diagram $A$ represents the partitions
%
%We refer to sums of Schur functions all with coefficient $1$ as `rugs' (imagining a
%Venn diagram like expression).  This shows that 
The disjoint sets of this sum can be grouped so that the sum is of only $\lfloor k/2 \rfloor +1$
terms, which shows that the coefficients   
always lie in the range $0$ through $\lfloor k/2 \rfloor+1$. The most interesting
aspect of this formula is that we see the lattice of subsets arising in a natural and
unexpected way in a representation theoretical
setting.  This is potentially part of a more general result and the hope is that this 
particular model will shed light on a general formula for the Kronecker product of
two Schur functions, but our main motivation for computing these is to develop
computational tools.

There are yet further motivations for restricting our attention to 
understanding the Kronecker product of $s_{(d,d)}$ with another Schur function.
% but there is reason to try to understand the Kronecker
%product with $s_{(d,d)}$ as peculiar and interesting sub-problem.
The Schur functions indexed by the partition $(d,d)$ are a special family for several 
combinatorial reasons and so there is reason to believe that their behavior will be more accessible than the general case for the Kronecker product of
two Schur functions.   More precisely, Schur functions
indexed by partitions with two parts are notable because they 
are the difference of two homogeneous
symmetric functions, for which a combinatorial formula for the Kronecker product is known.
In addition, a partition $(d,d)$ is rectangular and hence falls under a second category of
Schur functions that are often combinatorially more straightforward to manipulate than the general case.  
%There are hopes that a formula like Theorem~\ref{cleanest} can be generalized
%to a formula for $s_\lambda \ast s_\mu$ for all partitions $\lambda$ and $\mu$,
%but there remains much work to be done towards this goal.

%There are combinatorial reasons to think that there is something
%special about the operation of taking the Kronecker 
%product with the Schur function $s_{(d,d)}$. 
%---This has been said earlier. 
From the hook length formula
it follows that the number of standard tableau of shape $(d,d)$ is equal to
the Catalan number $C_{d} = \frac{1}{d+1} {{2d}\choose{d}}$.  Therefore, from the
perspective of $S_{2d}$ representations, taking the Kronecker product with the
Schur function $s_{(d,d)}$ and the Frobenius image of a module 
explains how the tensor of a representation with a particular irreducible 
module of dimension $C_d$ decomposes.

Work towards understanding the Kronecker product of two Schur functions
each indexed by a two-row partition was done by   Remmel and Whitehead \cite{RemmelWhitehead} and Rosas  \cite{Rosas} and we
specialize the results of the latter to obtain the the boolean lattice type expression
of Theorem \ref{cleanest}.  Kronecker product expressions in \cite{Rosas} involve adding and subtracting terms,  and do not particularly illuminate  the non-negativity of Kronecker coefficients that we demonstrate here. However,  the results in
\cite{Rosas} begin calculations necessary to arrive at some of the results 
we present.  

%She also
%provides a formula for the Kronecker product of a Schur function indexed by
%a hook and a two row function and the Kronecker product of two hook shapes.

%In section \ref{sec:tworow} we present the main result of the Kronecker product expression
%for $s_{(d,d)} \ast s_{(d+k, d-k)}$.  In section \ref{sec:hookexpr}, we present
%a formula for expression of the Kronecker products of the form $s_{(d,d)} \ast s_{(2d-k,1^k)}$
%where we show that the coefficients in this expression are all $1$ or $0$.
%In section \ref{sec:combinatorics} we demonstrate some combinatorial and 
%symmetric function consequences of this formula.

The paper is structured as follows. In the next section we review pertinent 
background information
including necessary symmetric function notation and lemmas necessary 
for later computation.
In Section~\ref{sec:tworow} we consider a generalization of formula 
\eqref{sddsdd}
to an expression for $s_{(d,d)} \ast s_{(d+k,d-k)}$.   In the following section 
(Section~\ref{sec:hookexpr}) 
we give an explicit combinatorial formula for the Kronecker product 
$s_{(d,d)} \ast s_{(2d-k, 1^k)}$.  We consider the hook case because it is an 
infinite family of Kronecker products that also seems to
have a compact formula.  We also observe in Theorem~\ref{the:hookform} that 
the product $s_{(d,d)} \ast s_{(2d-k, 1^k)}$ is multiplicity free. 
Finally, Section~\ref{sec:combinatorics} is devoted to  combinatorial and 
symmetric function consequences of our results.  In particular we are able to
give generating functions for the partitions that have a particular coefficient
in the expression $s_{(d,d)} \ast s_{(d+k,d-k)}$.

\section{Background}

\subsection{Partitions} 
A \emph{partition} of an integer $n$ is a finite sequence of non-negative integers
$(\lambda_1 \geq \lambda_2 \geq \cdots \geq \lambda_\ell)$ whose values sum
to $n$, denoted $\lambda \vdash n$.  The \emph{height} or \emph{length} of
the partition,   denoted $\ell(\lambda)$,  is the maximum index for which
$\lambda_{\ell(\lambda)} > 0$.  We call the $\lambda _i$ \emph{parts} or \emph{rows} of the
partition and if $\lambda _i$ appears $n_i$ times we abbreviate this subsequence to $\lambda _i^{n_i}$. With this in mind if $\lambda = (k^{n_k}, (k-1)^{n_{k-1}}, \ldots, 1^{n_1})$ then we define $z_\lambda = 1^{n_1}1! 2^{n_2}2!\cdots k^{n_k}k!  $.
%in abbreviated notation will write
%$(k^{n_k}, (k-1)^{n_{k-1}}, \ldots, 1^{n_1})$ where $n_i$ is the number of parts
%of the partition that are equal to $i$. 
The $0$ parts of the partition are optional
and we will assume that for $i > \ell(\lambda)$ we have $\lambda_i = 0$.  

    Let $\lambda=(\lambda_1, \lambda_2, \ldots, \lambda_\ell)$ be a partition of $n$.  To form the \emph{diagram} associated with $\lambda$, place a cell at each point in matrix notation $(j, i)$  where $1\leq i \leq \lambda_j$ and $1\leq j \leq \ell$.  We say $\lambda$ has \emph{transpose} $\lambda'$ if the diagram for $\lambda'$ is given by the points $(i,j)$ where $1\leq i \leq \lambda_j$ and $1\leq j \leq \ell$.
We call $\lambda \vdash n \geq 3$  a \emph{hook} if $\lambda=(\lambda_1, 1^m)$ where $\lambda_1 \geq 2$ and $m\geq 1$; if $m=0$, we call $\lambda$ a \emph{one-row shape}, and if $\lambda_1=1$, we call $\lambda$ a \emph{one-column shape}.  We say $\lambda=(\lambda_1,\lambda_2,2^{m_2}, 1^{m_1}) \vdash n \geq4$ is a \emph{double hook} if $\lambda_1 \geq \lambda_2 \geq 2$, $m_2 \geq0$, $m_1 \geq0$.  We say $\lambda$ is a \emph{two-row shape} if $\lambda=(\lambda_1,\lambda_2)$.

\subsection{Symmetric functions and the Kronecker product} The ring of symmetric functions is the graded subring of $\mathbb{Q}[x_1, x_2, \ldots]$ given by
$$\Lambda :=\mathbb{Q}[p_1, p_2, \ldots]$$where $p_i=x_1^i+x_2^i+\cdots$ are the elementary power sum symmetric functions. For $\lambda=(\lambda_1, \lambda_2, \ldots, \lambda_\ell)$ we define $p_\lambda := p_{\lambda _1}p_{\lambda _2}\cdots p_{\lambda _{\ell}}$. The interested reader
should consult a reference such as \cite{Mac} for more details of the structure of this ring.
It is straightforward to see that  $\{ p_\lambda \} _{\lambda \vdash n\geq 0}$ forms a basis for $\Lambda$. This basis is orthogonal to itself with respect to the scalar product on $\Lambda$:
\begin{equation*}
\langle p_\lambda , p_\mu \rangle = z_\lambda\delta _{\lambda \mu}.
\end{equation*}
However, our focus for this paper will be the basis of $\Lambda$ known as the basis of \emph{Schur functions}, $\{ s_\lambda \} _{\lambda \vdash n\geq 0}$, which is the orthonormal basis under this scalar product:
$$\langle s_\lambda , s_\mu \rangle = \delta _{\lambda \mu},$$
and its behaviour under the Kronecker product.

The \emph{Kronecker product} is the operation on symmetric functions
\begin{equation}\label{krondef}\frac{p_\lambda}{z_\lambda} \ast \frac{p_\mu}{z_\mu} 
= \delta_{\lambda\mu} \frac{p_\lambda}{z_\lambda}\end{equation}
that in terms of the Schur functions becomes 
$$s_\mu \ast s_\nu = \sum_{\lambda \vdash |\mu|} \kroncoeff_{\mu\nu\lambda} s_\lambda .$$
It transpires that the \emph{Kronecker coefficients} $\kroncoeff_{\mu\nu\lambda}$
encode the inner tensor product of symmetric group representations.
%or, dually, the restriction
%of $Gl_{nm}$ representations to $Gl_n \otimes Gl_m$ representations.
That is, if we denote the irreducible $S_n$ module indexed by a partition
$\lambda$ by $M^\lambda$, and $M^\mu \otimes M^\nu$ represents the
tensor of two modules with the diagonal action, then the module 
decomposes as
$$M^\mu \otimes M^\nu 
\simeq \bigoplus_\lambda  (M^\lambda)^{\oplus \kroncoeff_{\mu\nu\lambda}}.$$

The Kronecker coefficients also encode
the decomposition of $GL_{nm}$ polynomial 
representations to $GL_n \otimes GL_m$ representations
$$Res^{GL_{mn}}_{GL_n \otimes GL_m}(V^\lambda ) 
\simeq \bigoplus_{\mu, \nu} (V^\mu \otimes V^\nu)^{\oplus \kroncoeff_{\mu\nu\lambda} }.$$
It easily follows from  \eqref{krondef} and the linearity of the product that
these coefficients satisfy the symmetries
$$\kroncoeff_{\mu\nu\lambda}
=\kroncoeff_{\nu\mu\lambda}
=\kroncoeff_{\mu\lambda\nu}
=\kroncoeff_{\mu ' \nu ' \lambda},$$
and
$$\kroncoeff_{\lambda\mu(n)} = \kroncoeff_{\lambda\mu'(1^n)} =
\begin{cases}
1&\hbox{ if } \lambda = \mu\\
0 &\hbox{ otherwise}
\end{cases}
$$
that we will use extensively in what follows.

We will use some symmetric function identities in the remaining sections.  Recall that for
a symmetric function $f$,
the symmetric function operator $f^\perp$ (read `f perp') 
is defined to be the operator that is dual to
multiplication with respect to the scalar product.  That is,
\begin{equation}
\left< f^\perp g, h \right>  = \left< g, f \cdot h \right>~.
\end{equation}
The perp operator can also be defined linearly by
$$s_\lambda^\perp s_\mu = s_{\mu\slash \lambda} =
\sum_{\nu \vdash |\mu| - |\lambda|} c^{\mu}_{\lambda\nu} s_\nu$$
where the $c^{\mu}_{\lambda\nu}$ are the Littlewood-Richardson coefficients.
We will make use of the following well known relation, which connects the
internal and external products,
\begin{equation}\label{multperpkron}
\langle s_\lambda f, g \ast h \rangle = 
\sum_{\mu,\nu \vdash |\lambda|} \kroncoeff_{\lambda\mu\nu} \langle
f, (s_\mu^\perp g) \ast (s_\nu^\perp h) \rangle~.
\end{equation}

From these two identities and use of the Littlewood-Richardson rule we derive the following lemma.
\begin{lemma} \label{lem:recurrence}
If $\ell(\lambda) > 4$, then $\kroncoeff_{(d,d)(a,b)\lambda} = 0$.  Otherwise
it satisfies the following recurrences.
If $\ell(\lambda) =4$, then
\begin{equation}\label{fourcol}\left< s_{(d,d)} \ast s_{(d+k,d-k)}, s_\lambda \right>
= \left< s_{(d-2,d-2)} \ast s_{(d+k-2,d-k-2)}, s_{\lambda - (1^4)} \right>~.
\end{equation}
If $\ell(\lambda) =3$, then
\begin{align} \label{threecol}
\left< s_{(d,d)} \ast s_{(d+k,d-k)}, s_\lambda \right>
= &\left< s_{(d-1,d-1)} \ast s_{(d+k-1,d-k-1)},  s_{(1)} s_{\lambda - (1^3)} \right>\\
\nonumber&- \left< s_{(d-2,d-2)} \ast s_{(d+k-2,d-k-2)}, s_{(1)}^\perp s_{\lambda-(1^3)} \right>~.
\end{align}
If $\ell(\lambda) =2$, then
\begin{equation}\label{twocol}
\left< s_{(d,d)} \ast s_{(d+k,d-k)}, s_\lambda \right> = \begin{cases}
1&\hbox{ if }k \equiv \lambda_2~(mod~2)\hbox{ and }\lambda_2 \geq k\\
0&\hbox{ otherwise. }
\end{cases}
\end{equation}
\end{lemma}

\begin{proof}
For \eqref{twocol} we refer the reader to \cite[Theorem 2.2]{GWXZ1}.  The proof
there is similar to the proofs of \eqref{fourcol} and \eqref{threecol} that we provide here (in fact,
 \eqref{fourcol} and \eqref{threecol} with $k=0$ appear in that reference).

If $\ell(\lambda)>4$, $\kroncoeff_{(d,d)(a,b)\lambda}$ is the difference
of $\langle s_{(d)} s_{(d)}, s_{(a,b)} \ast s_\lambda \rangle$ and
   $\langle s_{(d+1)} s_{(d-1)}, s_{(a,b)} \ast s_\lambda \rangle$.
By \eqref{multperpkron}, $\langle s_{(r)} s_{(s)}, s_{(a,b)} \ast s_\lambda \rangle
= \sum_{\mu \vdash r} \langle s_{(a,b) \slash \mu}, s_{\lambda\slash\mu} \rangle$.
The Littlewood-Richardson rule says that the terms in the expansion  
of $s_{(a,b) \slash \mu}$
in the Schur basis will have length at most $2$ and this $s_{(a,b)  
\slash \mu}$ will be zero unless the
length of $\mu$ is less than or equal to 2.  By consequence, the terms
in the expansion of
$s_{\lambda\slash\mu}$ in the Schur basis will be indexed by  
partitions of length greater than two,
therefore this expression will always be $0$.

Assume that $\ell(\lambda) = 4$.
By the Pieri rule we have that $s_{(1^4)} s_{\lambda - (1^4)} = s_\lambda + $
terms of the form $s_\gamma$ where $\ell(\gamma) > 4$.  By consequence,
\begin{align*}\langle s_{(d,d)} \ast s_{(d+k,d-k)}, s_\lambda \rangle &=
\langle s_{(d,d)} \ast s_{(d+k,d-k)}, s_{(1^4)} s_{\lambda - (1^4)} \rangle\\
&= \sum_{\mu \vdash 4} \langle (s_\mu^\perp s_{(d,d)}) \ast (s_{\mu'}^\perp s_{(d+k,d-k)}), 
s_{\lambda - (1^4)} \rangle ~.
\end{align*}
Every term in this sum is $0$ unless both $\mu$ and $\mu'$ have length less than or equal to $2$.  
The only term for which this is true is 
$\mu = (2,2)$ and $s_{(2,2)}^\perp( s_{(a,b)} ) = s_{(a-2,b-2)}$, hence \eqref{fourcol} holds.

Assume that $\ell(\lambda) = 3$.  Although there are cases to check, it follows again
from the Pieri rule that $s_{(1^3)} s_{\lambda - (1^3)} - s_{(1^4)} s_{(1)}^\perp s_{\lambda - (1^3)}
= s_\lambda +$ terms involving $s_\gamma$ where $\ell(\gamma)> 4$.  Therefore,
\begin{align*}
\langle s_{(d,d)} \ast s_{(d+k,d-k)}, s_\lambda \rangle &=
\langle s_{(d,d)} \ast s_{(d+k,d-k)}, s_{(1^3)} s_{\lambda - (1^3)} - s_{(1^4)} s_{(1)}^\perp s_{\lambda - (1^3)} \rangle\\
& = \sum_{\mu \vdash 3} \langle (s_\mu^\perp s_{(d,d)}) \ast (s_{\mu'}^\perp s_{(d+k,d-k)}), s_{\lambda - (1^3)} \rangle \\
&\hskip .3in- \langle s_{(d-2,d-2)} \ast s_{(d+k-2,d-k-2)}, s_{(1)}^\perp s_{\lambda - (1^3)} \rangle~.
\end{align*}
Again, in the sum the only terms that are not equal to $0$ are those that have
the length of both $\mu$ and $\mu'$ less than or equal to $2$ and in this case
the only such partition is $\mu = (2,1)$.  By the Littlewood-Richardson rule we have that
$s_{(2,1)}^\perp( s_{(a,b)} ) = s_{(1)}^\perp( s_{(a-1,b-1)} )$, hence this last expression is
equal to
\begin{align*}
\langle &(s_{(1)}^\perp s_{(d-1,d-1)}) \ast (s_{(1)}^\perp s_{(d+k-1,d-k-1)}), s_{\lambda - (1^3)} \rangle
- \langle s_{(d-2,d-2)} \ast s_{(d+k-2,d-k-2)}, s_{(1)}^\perp s_{\lambda - (1^3)} \rangle\\
&= \langle s_{(d-1,d-1)} \ast  s_{(d+k-1,d-k-1)}), s_{(1)} s_{\lambda - (1^3)} \rangle
- \langle s_{(d-2,d-2)} \ast s_{(d+k-2,d-k-2)}, s_{(1)}^\perp s_{\lambda - (1^3)} \rangle~.
\end{align*}
\end{proof}

\comment{
\subsection{Characteristic functions} }
To express our main results we will use the characteristic of a boolean valued proposition. If $R$ is a proposition then we denote the propositional characteristic (or indicator) function of $R$ by
        \begin{eqnarray*}
            \charf{ R }=  \left\{
                \begin{array}{ll}
                    1 & \textrm{if proposition R is true}
                    \\ 0 & \textrm{otherwise.}
                \end{array}
            \right.
        \end{eqnarray*}
\comment{

These characteristic functions have the following basic properties.
\begin{lemma} \label{propfunc}
Let $Q$ and $R$ be propositions, and $\sim \! R$ denote the negation of $R$.
  
    {\small
    \begin{tabular}{lcl} 
        1. $\Big. \charf{ R   }\charf{ \! \sim \! R  }=0$. \Big. & \hspace{1cm}&
        4. $\charf{ R \textrm{ or } Q  }
            = \charf{ R   } + \charf{ Q   }
                - \charf{ R   }\charf{ Q   }.$ \\
        2. $\Big. \charf{ R   }+\charf{ \! \sim \! R  }=1.\Big. $ &&
        5. If $R \Rightarrow Q$, then $\charf{ R   }=\charf{ R   }
            \charf{ Q   } \neq \charf{ Q  }$. \\
        3. $\Big. \charf{ R \textrm{ and } Q  }
            = \charf{ R   }\charf{ Q   }.\Big. $ &&
        6. If $R \Leftrightarrow Q$, then $\charf{ R   }=\charf{ R   }
            \charf{ Q   }=\charf{ Q   }.  $
    \end{tabular}}
\end{lemma}

Characteristic functions also give us some useful ways to write $\left\lceil \frac{t}{2} \right\rceil$.
\begin{lemma} \label{ceil lemma}
    Let $t$ be a non-negative integer.  Then
    $$         \left\lceil \frac{t}{2} \right\rceil  = \frac{t + \charf{ t \textrm{ odd}  }}{2}
         = \left\lceil \frac{t\! +\! 1}{2} \right\rceil
            - \charf{ t \textrm{ even}  }
         = \sum_{ \stackrel{1 \leq \, a  }{ a \; odd}}
                \charf{ a \leq t }.
    $$
\end{lemma}
}

\begin{section}{The Kronecker product $s_{(d,d)} \ast s_{(d+k,d-k)}$} \label{sec:tworow}

Let $P$ indicate the set of partitions with four even parts or four odd parts and let
$\overline{P}$ be the set of partitions with 2 even parts and 2 odd parts (the
complement of $P$ in the set of partitions with at most 4 parts).

We let $\gamma P$ represent the set of
partitions $\lambda$ of $2d$ (the value of $d$ will be implicit in the
left hand side of the expression)
such that $\lambda - \gamma \in P$.  We also allow
$( \gamma \uplus \alpha) P =  \gamma P \cup \alpha P$.  Under all cases that we will
consider the partitions in $\gamma P$ and $\alpha P$ are disjoint.

\begin{theorem}\label{cleanest}
Let $\lambda$ be a partition of $2d$,
\begin{align} \label{cleaneq}
\left< s_{(d,d)} \ast s_{(d+k,d-k)}, s_\lambda \right> =
\sum_{i=0}^k \charf{ \lambda  \in (k{+}i,k,i) P }+ 
\sum_{i=1}^k \charf{ \lambda  \in (k{+}i{+}1,k{+}1,i) P}~.
\end{align}
\end{theorem}

Because of the notation we have introduced, Theorem
\ref{cleanest} can easily be restated  in the following corollary.

\begin{corollary}\label{rugsum}
For $k \geq 0$, Theorem \ref{cleanest} is equivalent to:
if $k$ is odd then
\begin{align*}
s_{(d,d)} \ast s_{(d+k,d-k)} = 
\rug{ ((k,k) &\uplus (k+1,k,1) \uplus (k+2,k+1,1))P } \\+ &\sum_{i=1}^{(k-1)/2} \rug{ ((k+2i,k,2i) \uplus 
(k+2i+1,k+1,2i) \uplus\\
&(k+2i+1,k,2i+1) \uplus (k+2i+2,k+1,2i+1))P}
\end{align*}
and if $k$ is even then
\begin{align*}
s_{(d,d)} \ast s_{(d+k,d-k)} = 
\rug{ (k,k)P } + \sum_{i=1}^{k/2} \rug{ &((k+2i-1,k,2i-1) \uplus 
(k+2i,k+1,2i-1) \\&\uplus (k+2i,k,2i) \uplus (k+2i+1,k+1,2i))P}~.
\end{align*}
As a consequence, the coefficients of $s_{(d,d)} \ast s_{(d+k,d-k)}$ will all
be less than or equal to $\lfloor k/2 \rfloor + 1$.  
\end{corollary}
\begin{remark}
Note that the upper bound on the coefficients that appear in the expressions
$s_{(d,d)} \ast s_{(d+k,d-k)}$ are sharp in the sense that for sufficiently large $d$,
there is a coefficient that will be equal to $\lfloor k/2 \rfloor + 1$.
\end{remark}
\begin{remark}
We note that this regrouping of the rugs is not unique but it is useful because there are partitions
that will fall in the intersection of each of these sets.  This set of rugs is
also not unique in that it is possible to describe other collections of sets of
partitions (e.g. see \eqref{eq:ke2}).
% as a matter of fact, we had another whole theorem which has its own
% strengths, but we cut it out
\end{remark}

We note that for $k=2$, the expression
stated in the introduction does not exactly follow this decomposition but it does follow
from some manipulation.

\begin{corollary} \label{cor:easycases}  For $d \geq 1$,
\begin{equation}\label{eq:ke0}s_{(d,d)} \ast s_{(d,d)} = \rug{ P } \end{equation}
\begin{equation}\label{eq:ke1}s_{(d,d)} \ast s_{(d+1,d-1)} = \rug{ \overline{P}}\end{equation}
and for $d \geq 2$,
\begin{equation}\label{eq:ke2}s_{(d,d)} \ast s_{(d+2,d-2)} = \rug{ P \cap \hbox{\rm no three parts are equal}} +
\rug{\hbox{\rm distinct partitions} }~.\end{equation}
\end{corollary}
\begin{proof}  Note that  \eqref{eq:ke0} is just a restatement of Corollary \ref{rugsum}
in the case that $k=0$ and is \cite[Theorem I.1]{GWXZ1}.

First, by Theorem \ref{cleanest} we note that 
$$\kroncoeff_{(d,d)(d+1,d-1)\lambda} =
\charf{ \lambda \in (1,1)P} + \charf{ \lambda \in (2,1,1)P} + \charf{ \lambda \in (3,2,1)P}~.$$
If $\lambda \in P$, then $\lambda - (1,1)$, $\lambda - (2,1,1)$, $\lambda - (3,2,1)$ are all
not in $P$ so each of the terms in that expression are $0$.
If $\lambda$ is a partition with two even parts and two odd parts (that is, $\lambda \in
\overline{P}$)
then either $\lambda_1 \equiv
\lambda_2$ and $\lambda_3 \equiv \lambda_4~(mod~2)$ or $\lambda_1 \equiv
\lambda_3$ and $\lambda_2 \equiv \lambda_4~(mod~2)$ or $\lambda_2 \equiv
\lambda_3$ and $\lambda_1 \equiv \lambda_4~(mod~2)$.  In each of these three cases,
exactly one of the expressions $\charf{ \lambda \in (1,1)P}$, 
$\charf{ \lambda \in (2,1,1)P}$ or $\charf{ \lambda \in (3,2,1)P}$ will be $1$ and the other
two will be zero.  Therefore,
$$\sum_{\lambda \vdash 2d} \kroncoeff_{(d,d)(d+1,d-1)\lambda} s_\lambda = 
\rug{ \overline{P} }~.$$

We also have by Theorem \ref{cleanest} that
\begin{align*}C_{(d,d)(d+2,d-2)\lambda} = &\charf{ \lambda \in (2,2)P}
+\charf{ \lambda \in (4,2,2)P}-\charf{ \lambda \in (6,4,2)P}\\
&+\charf{ \lambda \in (3,2,1)P}
+\charf{ \lambda \in (4,3,1)P}+\charf{ \lambda \in (5,3,2)P}+\charf{ \lambda \in (6,4,2)P}~.
\end{align*}

Any distinct partition that is in $P$ is also in $(6,4,2)P$.  Every distinct
partition that is in $\overline{P}$ will have two odd parts and two even parts and
will be in one of $(3,2,1)P$, 
$(4,3,1)P$, or $ (5,3,2)P$ depending on which of $\lambda_2, \lambda_1$ or $\lambda_3$
is equal to $\lambda_4$ modulo 2 (respectively).  Therefore, we have
\begin{equation}
\rug{\hbox{ distinct partitions }} = \rug{ ((3,2,1) \uplus (4,3,1) \uplus (5,3,2) \uplus (6,4,2)) P}~.
\end{equation}

%For every partition $\lambda \in P$ which does not have 3 equal parts either has
%$\lambda_2>\lambda_3$ or it has $\lambda_1 > \lambda_2$ and $\lambda_3 > \lambda_4$ or
%both of these conditions hold and the partition is distinct.  If $\lambda_2 > \lambda_3$, then
%$\lambda \in (2,2)P$.  If  $\lambda_1 > \lambda_2$ and $\lambda_3 > \lambda_4$, then
%$\lambda \in (4,2,2)P$.  Since $(6,4,2)P$ is the set of distinct partitions in $P$, we
%can remove the c
If $\lambda \in (2,2)P \cap (4,2,2)P$, then $\lambda_2 \geq \lambda_3+2$ because $\lambda \in (2,2)P$,
and $\lambda_1 \geq \lambda_2 + 2$
  and $\lambda_3 \geq \lambda_4+2$ because $\lambda \in (4,2,2)P$,
 so $\lambda \in (6,4,2)P$.  Conversely, one verifies that in fact
 $(2,2) P \cap (4,2,2) P = (6,4,2) P$, hence
 $$\rug{ (2,2)P \cup (4,2,2) P } = \rug{(2,2)P} + \rug{ (4,2,2) P } - \rug{ (6,4,2) P}~.$$
If $\lambda \in P$ does not have three equal parts, then either $\lambda_2 > \lambda_3$ or 
$\lambda_1 > \lambda_2$ and $\lambda_3 > \lambda_4$.  Therefore, $\lambda \in (2,2)P \cup
(4,2,2)P$ and  hence $(2,2)P \cup (4,2,2) P = P \cap \hbox{ no three parts are equal }$.
\end{proof}

\begin{proof} (of Theorem \ref{cleanest})
Our proof proceeds by induction on the value of $d$
and uses the Lemma \ref{lem:recurrence}.  We will consider
two base cases because \eqref{threecol} and \eqref{fourcol} give
recurrences for two smaller values
of $d$.  The exception for this is of course that $\lambda$ is a partition of length $2$
since it is easily verified that the two sides of  \eqref{cleaneq} agree since the
only term on the right hand side of the equation that can be non-zero is $\charf{ \lambda \in 
(k,k)P}$.

When $d=k$ and we have that the left hand side of \eqref{cleaneq}
is $\left< s_{(k,k)} \ast s_{(2k)}, s_\lambda \right>$, which is $1$ if $\lambda = (k,k)$
and $0$ otherwise.  On the right hand side of \eqref{cleaneq} we have
that $\charf{ \lambda  \in (k,k)P }$ is $1$ if and only if $\lambda = (k,k)$
and all other terms are $0$ and hence the two expressions agree.

If $d=k+1$, then $s_{(k+1,k+1)} \ast s_{(2k-1,1)} = s_{(k+1,k,1)} + s_{(k+2,k)}$.  Notice
that the only partitions $\lambda$ of $2k+2$ such that the indicator functions on the
right hand side of \eqref{cleaneq} can be satisfied are $\charf{ \lambda  \in (k,k)P }$
when $\lambda = (k{+}2,k)$ and $\charf{ \lambda  \in (k{+}1,k,1)P }$ when $\lambda = (k{+}1,k,1)$.
All others must be $0$ because the partitions that are subtracted off are larger than $2k{+}2$.

Now assume that \eqref{cleaneq} holds for all values strictly smaller than $d$.  If 
$\ell(\lambda) = 4$, then $\lambda - \gamma \in P$ if and only if $\lambda - \gamma - (1^4) \in P$
for all partitions $\gamma$ of length less than or equal to $3$ so
\begin{align*}
\left< s_{(d,d)} \ast s_{(d+k,d-k)}, s_\lambda \right>
&= \left< s_{(d-2,d-2)} \ast s_{(d+k-2,d-k-2)}, s_{\lambda - (1^4)} \right>\\
&= \sum_{i=0}^k \charf{ \lambda{-}(1^4) \in (k{+}i,k,i)P }+ \sum_{i=1}^k \charf{ \lambda{-}(1^4) \in 
(k{+}i{+}1,k{+}1,i)P }\\
&= \sum_{i=0}^k \charf{ \lambda  \in (k{+}i,k,i)P }+ 
\sum_{i=1}^k \charf{ \lambda  \in (k{+}i{+}1,k{+}1,i)P }~.
\end{align*}

So we can now assume that $\ell(\lambda) = 3$.  
By  \eqref{threecol} we need to consider the coefficients of the form
$\left< s_{(d,d)} 
\ast s_{(d+k,d-k)}, s_\mu \right>$ where
$s_\mu$ appears in the expansion of 
$s_{(1)} s_{\lambda - (1^3)}$ 
or $s_{(1)}^\perp s_{\lambda-(1^3)}$.
If $\lambda$ has three distinct parts and $\lambda_3 \geq 2$
then $\mu  = \lambda - \delta$ where 
$\delta \in \{  (1,1,0), (1,0,1), (0, 1, 1),$ 
$(1,1,1,-1), (2,1,1), (1,2,1), (1,1,2) \}$
and we can assume by induction that these expand into terms
of the form $\pm\charf{ \lambda - \delta - \gamma \in P }$ where $\gamma$ is a
partition.  However, if $\lambda$ is not distinct or
$\lambda_3 = 1$, then for some $\delta$ in the set,  
$\lambda {-} \delta$ will not be a partition and 
$\charf{ \lambda {-} \delta  \in \gamma P}$ will be $0$ 
and we can add these terms to our formulas so that we can treat
the argument uniformly and not have to consider different 
possible $\lambda$.

One obvious reduction we can make to treat the expressions more
uniformly is to note that $\charf{ \lambda {-} (1,1,1,{-}1)  \in \gamma P} = 
\charf{ \lambda {-} (2,2,2)  \in \gamma P}$.

Let $C_1 = \{  (1,1,0), (1,0,1), (0, 1, 1), (2,2,2)\}$ and
$C_2 = \{  (2,1,1), (1,2,1), (1,1,2)\}$.  By the induction hypothesis  and 
\eqref{threecol} we have
that $\left< s_{(d,d)} \ast s_{(d+k,d-k)}, s_\lambda \right>$ is equal to
\begin{align*}
&\sum_{\delta \in C_1}
\Big( \sum_{i=0}^k \charf{ \lambda {-} \delta  \in (k{+}i,k,i)P }+ 
 \sum_{i=1}^k \charf{ \lambda {-} \delta  \in (k{+}i{+}1,k{+}1,i)P } \Big)\\
&- \sum_{\delta \in C_2}\Big( \sum_{i=0}^k \charf{ \lambda {-} \delta  \in (k{+}i,k,i)P }+ 
\sum_{i=1}^k \charf{ \lambda {-} \delta  \in (k{+}i{+}1,k{+}1,i)P } \Big)~.
\end{align*}

We notice that $\lambda - (2,2,2) - (k{+}i,k,i) = \lambda - (1,1,2) - (k{+}i{+}1,k{+}1,i)$
and $\lambda - (2,1,1) - (k{+}i,k,i) = \lambda - (1,0,1) - (k{+}i{+}1,k{+}1,i)$ and
$\lambda - (1,2,1) - (k{+}i,k,i) = \lambda - (0,1,1) - (k{+}i{+}1,k{+}1,i)$ so the corresponding
terms always cancel.  With this reduction, we are left with the terms
\begin{align*}
&\sum_{\delta \in C_3}
 \sum_{i=0}^k \charf{ \lambda {-} \delta  \in (k{+}i,k,i)P }
+ \sum_{\delta \in C_4}\sum_{i=1}^k \charf{ \lambda {-} \delta  \in (k{+}i{+}1,k{+}1,i)P } \\
&- \sum_{i=0}^k \charf{ \lambda {-} (1,1,2)  \in (k{+}i,k,i)P}
- \sum_{\delta \in C_5} \sum_{i=1}^k \charf{ \lambda {-} \delta  \in (k{+}i{+}1,k{+}1,i)P } \\
&- \charf{ \lambda {-} (1,2,1)  \in (k,k)P } 
- \charf{ \lambda {-} (2,1,1)  \in (k,k)P }
+ \charf{ \lambda{-}(2,2,2) \in (k,k)P }
\end{align*}
where $C_3 = \{ (0,1,1), (1,0,1), (1,1,0) \}$, $C_4 = \{ (1,1,0), (2,2,2) \}$, $C_5 = \{ (2,1,1), (1,2,1) \}$.

Next we notice that 
$\lambda - (1,1,2) - (k{+}i,k,i) = \lambda - (0,1,1) - (k{+}i+1,k,i{+}1)$,
$\lambda - (2,1,1) - (k{+}i{+}1,k{+}1,i) = \lambda - (1,1,0) - (k{+}i{+}2,k{+}1,i{+}1)$, and 
$\lambda - (2,2,2) - (k{+}i{+}1,k{+}1,i) = \lambda - (1,2,1) - (k{+}i{+}2,k{+}1,i{+}1)$.  Then by canceling
these terms and joining the compositions that are being subtracted off in the sum, these
sums reduce to the following expression.

%\begin{align*}
%&\sum_{i=0}^k \charf{ \lambda {-} (1,0,1) {-} (k{+}i,k,i) \in P }+
% \sum_{i=0}^k \charf{ \lambda {-} (1,1,0) {-} (k{+}i,k,i) \in P }\\
%&+\charf{ \lambda {-} (0,1,1) {-} (k,k) \in P }
%+\charf{ \lambda {-} (1,1,0) {-} (k{+}2,k{+}1,1) \in P } \\
%&+\charf{ \lambda {-} (2,2,2) {-} (2k{+}1,k{+}1,k) \in P }
%+ \charf{ \lambda{-}(2,2,2){-}(k,k) \in P }\\
%&-\charf{ \lambda {-} (1,1,2) {-} (2k,k,k) \in P }
%- \charf{ \lambda {-} (2,1,1) {-} (2k{+}1,k{+}1,k) \in P }\\
%&- \charf{ \lambda {-} (1,2,1) {-} (k{+}2,k{+}1,1) \in P }
%- \charf{ \lambda {-} (1,2,1) {-} (k,k) \in P } \\
%&- \charf{ \lambda {-} (2,1,1) {-} (k,k) \in P }
%\end{align*}

\begin{align*}
&\sum_{i=0}^k \charf{ \lambda  \in (k{+}i{+}1,k,i{+}1)P }+
 \sum_{i=0}^k \charf{ \lambda  \in (k{+}i{+}1,k{+}1,i)P }\\
&+\charf{ \lambda  \in (k,k{+}1,1) P}
+\charf{ \lambda  \in (k{+}3,k{+}2,1) P} \\
&+\charf{ \lambda  \in (2k{+}3,k{+}3,k{+}2)P}
+ \charf{ \lambda\in (k{+}2,k{+}2,2) P }\\
&-\charf{ \lambda  \in (2k{+}1,k{+}1,k{+}2) P}
- \charf{ \lambda  \in (2k{+}3,k{+}2,k{+}1) P}\\
&- \charf{ \lambda  \in (k{+}3,k{+}3,2)P }
- \charf{ \lambda  \in  (k{+}1,k{+}2,1)P} \\
&- \charf{ \lambda  \in (k{+}2,k{+}1,1) P}~.
\end{align*}

Since $\ell(\lambda)=3$, 
if $\lambda - (a,b) \in P$, then $\lambda_1-a \geq \lambda_2-b \geq \lambda_3 \geq 2$, 
which is true if and only if $\lambda_1-a-2 \geq \lambda_2-b-2 \geq \lambda_3-2 \geq 0$.  
In particular,
$\charf{ \lambda \in (k{+}2,k{+}2,2) P } = \charf{ \lambda \in (k,k) P}$ and 
$\charf{ \lambda \in (k{+}3,k{+}3,2) P } = \charf{ \lambda  \in (k{+}1,k{+}1)P }$.

By verifying a few conditions it is easy to check that $\lambda  \in (r,s,s+1)P$ if
and only if $\lambda \in (r+2,s+2,s+1)P $ and similarly $\lambda  \in (s,s+1,r)P$
if and only if $\lambda  \in (s+2,s+1,r) P$.  With this relationship, we have the equivalence
of 
$\charf{ \lambda  \in (k,k{+}1,1) P} = \charf{ \lambda  \in (k{+}2,k{+}1,1) P}$,
$\charf{ \lambda  \in (k{+}1,k{+}2,1) P } = \charf{ \lambda  \in (k{+}3,k{+}2,1) P}$,
$\charf{ \lambda  \in (2k{+}1,k{+}1,k{+}2) P } = \charf{ \lambda  \in (2k{+}3,k{+}3,k{+}2) P}$,
and
$ \charf{ \lambda  \in (2k{+}1,k,k{+}1)P } = \charf{ \lambda \in (2k{+}3,k{+}2,k{+}1)P }$.
Each of these appear in the expression above.  After we cancel these terms the expression
reduces to
\begin{align*}
&\sum_{i=0}^{k-1} \charf{ \lambda  \in (k{+}i{+}1,k,i{+}1) P}+
 \sum_{i=1}^k \charf{ \lambda \in (k{+}i{+}1,k{+}1,i) P }+ \charf{ \lambda \in (k,k) P}~.
\end{align*}

This concludes the proof by induction on $d$ since we know the identity holds for each partition
$\lambda$ of length $2, 3$ or $4$.
\end{proof}
\end{section}

\begin{section}{The Kronecker product $s_{(d,d)} \ast s_{(2d-\! k,1^k)}$}\label{sec:hookexpr}

We could not visit the problem of the Kronecker product with $s_{(d,d)}$
without considering formulas that can be derived from previously known results
for Kronecker products with a two-row Schur function \cite{BOrellana2, RemmelWhitehead, Rosas}.
Extracting a completely positive formula is somewhat of a challenge, but possible
in this case because we have chosen to restrict our attention to a particular two-row shape.  Consider
the following result of Rosas \cite{Rosas}.
%\todo{ Are we sure that we can't get this from BOrellana2?}

    \begin{proposition} \label{Rosas1}\cite[Theorem 4]{Rosas}
    %this is a proposition and not a theorem because it is not our result
        Let $\mu, \nu$,\ and $\lambda$ be partitions of n, where $\mu=(\mu_1,\mu_2)$ 
        is a two-row shape, and $\nu=(\nu_2,1^{\nu_1})$ is a hook.  
        \begin{enumerate} \label{hook form}
            \item If $\lambda$ is a one-row shape, then 
            $\kroncoeff_{\mu\nu\lambda}\! = \!(\!( \mu \! \!= \!\! \nu \! )\!)$.
            \item If $\lambda=(\lambda_1,1^m)$ is a hook, then
                $$\kroncoeff_{\mu\nu\lambda}=\charf{ \mu_2 \! - \! 1 \leq m \leq \mu_1  \! }
                    \charf{ m \! = \! \nu_1  \! }
                   + \charf{ 2\mu_2 \leq m \! + \! \nu_1 \!+ \! 1 \leq 2\mu_1  \! }
                    \charf{ |m \! - \! \nu_1| \leq 1  \! }.$$

            \item If $\lambda=(\lambda_1,\lambda_2,2^{m_2},1^{m_1})$ is a double hook with $\lambda_1 \!-\! \lambda_2 \leq m_1$, then
                \begin{eqnarray*}
                   \kroncoeff_{\mu\nu\lambda} &=&  \charf{ \lambda_2 \leq \mu_2 \! - \! m_2 \leq \lambda_1 \! }
                        \charf{ 0 \leq \nu_1 \! - \! m_1 \! - \! 2m_2 \leq 3 \! }\\
                   &&  \charf{ \lambda_2 \leq \mu_2 \! - \! m_2 \! - \! 1\leq \lambda_1 \! }
                        \charf{ 1 \leq \nu_1 \! - \! m_1 \! - \! 2m_2 \leq 2 \! }\\
                   && \charf{ \lambda_2 \leq \mu_2 \! - \! m_2 \! + \! 1 \leq \lambda_1 \! }
                        \charf{ 1 \leq \nu_1 \! - \! m_1 \! - \! 2m_2 \leq 2 \! }\\
                   && - \charf{ \lambda_2 \! + \! m_2 \! + \! m_1 \! =\! \mu_2}
                        \charf{ 1 \leq \nu_1 \! - \! m_1 \! - \! 2m_2 \leq 2 \! }.
                \end{eqnarray*}
            \item If $\lambda$ is not contained in a double hook, then $\kroncoeff_{\mu\nu\lambda}=0$.
        \end{enumerate}
    \end{proposition}
    
\begin{remark} In the case that $\lambda$ is a double hook and
$\lambda_1 - \lambda_2>m_1$, then this formula applies by using the
transpose symmetry of the Kronecker coefficients
$\kroncoeff_{\mu\nu\lambda} = \kroncoeff_{\mu\nu'\lambda'}$.
\end{remark}

From this result we can arrive at a useful recursive formula.

\comment{
\begin{remark}\label{rem:transposesymmetry} We see from Theorem~\ref{Rosas1} part 1 that $\kroncoeff_{\mu  \nu\: (n) }=0$ unless $\nu=\mu$, so if $\lambda=(1,1,\ldots,1)$ then $\lambda'=(n)$ and
$$ \kroncoeff_{\mu\nu\lambda}=\kroncoeff_{\mu \nu' \: (n)}=\!(\!( \mu \! \!= \!\! \nu'  )\!).$$\end{remark}

We now aply apply this and Theorem~\ref{Rosas1}  to the case when $\mu=(d,d)$, and $\nu=(2d-k,1^k)$.

\begin{proposition} \label{ddcase}
    Let $\lambda$, $\mu$, and $\nu$ be partitions of $2d$, where $1\leq d$ and $0 \leq k \leq 2d  - \! 1$, $\nu=(2d\! -\! k,1^k)$, and $\mu=(d,d)$.
    \begin{enumerate}
        \item
            If $\lambda=(\lambda_1,1^m)$ is a hook, then
            $$\kroncoeff_{\mu\nu\lambda}=\charf{ d \! - \! 1 \! \leq \! m \! \leq d }
                \charf{ d \! - \! 1 \! \leq \! k \! \leq d }.$$
        \item
            Let $\lambda=(\lambda_1,\lambda_2,2^{m_2},1^{m_1})$.  If $\lambda$ is a double hook  or $\nu$ is a one-row shape or a one-column shape, then
        \begin{eqnarray*}
            \kroncoeff_{\mu\nu\lambda} & =& \charf{ 1 \! \leq \! k \! - \! m_1 \! - \! 2m_2 \! \leq \! 2  } \!
            \rug{ \charf{d \!  = \! \lambda_1\! +\! m_2 }
            \charf{1 \! \leq \!  m_1 }
            + \charf{d \! \pm \! 1 \! = \! \lambda_1\! +\! m_2 } } \\
            && +  \charf{ k \! - \! m_1 \! - \! 2m_2 \! = \! 0 \textrm{ or } 3 }
            \charf{d\! =\! \lambda_1\! +\! m_2 }.
        \end{eqnarray*}
        \item If $\lambda$ is a one-row shape, one-column shape, or is not contained in a double hook, then $\kroncoeff_{\mu\nu\lambda}=0$.
    \end{enumerate}
\end{proposition}

\begin{remark}\label{rem:tworowrosas} In Proposition~\ref{ddcase} part 2 when $\nu$ is a one-row shape or one-column shape, $\lambda=(\lambda_1,\lambda_2,2^{m_2},1^{m_1})$ need not be a double hook, that is if $\lambda_2=1$, then $m_2=0$, and if $\lambda_2=0$, then $m_2=m_1=0$.\end{remark}

\begin{proof} Let $1\leq k \leq 2d\!-\!2$ and $2\leq d$.  Then $\nu$ is a hook, so we apply Theorem \ref{Rosas1} to the three cases outlined in Proposition ~\ref{ddcase}.
    \begin{enumerate}
        \item Let $\lambda=(\lambda_1,1^m)$ be a hook.  From Theorem \ref{Rosas1} part 2 with $\mu_1=\mu_2=d$ and $\nu_1=k$, we have that
        \begin{eqnarray*}
        \kroncoeff_{\mu\nu\lambda} &=& \Big. \charf{ d\! -\! 1 \!\leq \! k \!\leq\! d } \!
            \charf{ k\! =\! m }
            + \charf{ 2d\! = k \! + \! m \! + \! 1 }
             \charf{ |k\! -\! m|\! \leq \! 1 } \Big.
\\
         &=& \Big. \charf{ d\! -\! 1 \!\leq \! k \!\leq\! d } \!
            \charf{ k\! =\! m }
            +  \charf{ 2d\!=\! 2m \! + \! (k\!-\!m)\!+\!1 } \!\!\Big. \!\!
                \sum_{ -1 \leq a \leq 1 } \!\!\!\! \charf{ k\! -\! m\!= \! a }
\\
        &=& \Big[ \! \charf{ d\! -\! 1 \! = \! k } \! + \! \charf{ d \! = \! k }\! \Big] \!
            \charf{ k\! =\! m }
            + \charf{ k\! -\! m\!= \! 0 } \! \charf{ 2d\!=\! 2m \! + \! 1 }\\
        &&  + \Big. \charf{ k\! -\! m\!= \! 1 } \! \charf{ 2d\!=\! 2m \! + \! 2 }
            + \charf{ k\! -\! m\!= \! - \! 1 } \! \charf{ 2d\!=\! 2m } \Big.
\\
        &=& \Big. \charf{ d\! -\! 1 \! = \! k } \! \charf{ d\! -\! 1\! =\! m }
            + \! \charf{ d \! = \! k } \! \charf{ d\! =\! m } \Big.\\
         && + \charf{ k\! = \! d } \! \charf{ d\! - \! 1\! =\! m }
            + \charf{ k\! = \! d \! - \! 1 } \! \charf{ d\!=\! m }
\\
        & =&\bigg. \Big[ \charf{ d \! = \! m } \! +  \! \charf{ d \! - \! 1 \! = \! m }\Big] \!\bigg.
            \Big[ \charf{ d \! = \! k } \! +  \! \charf{ d \! - \! 1 \! = \! k }\Big].
        \end{eqnarray*}

        \item We begin by showing the formula from Theorem \ref{Rosas1} part 2 gives the required result when $\nu$ is a one-row shape or one-column shape, and then show it when $\lambda$ is a double hook.

        Let $k=0$ and $1\leq d$. Then $\nu=(2d)$ is a one-row shape, so from Remark~\ref{rem:transposesymmetry} and the symmetry of Kronecker coefficients, $$\kroncoeff_{\mu\, (2d) \, \lambda} = \kroncoeff_{\mu \lambda \, (2d) }=\charf{ \lambda \!=\! \mu \!=\! (d,d) }.$$  We also have that $\charf{ 1 \! \leq \! 0 \! - \! m_1 \! - \! 2m_2 \! \leq \! 2  }=0$, so the formula from part 2 of the proposition gives that
        \begin{eqnarray*}
            \kroncoeff_{\mu\, (2d) \, \lambda}&=&\charf{ 0 \! - \! m_1 \! - \! 2m_2 \! = \! 0 \textrm{ or } 3 \! }
                \charf{\lambda_1\! +\! m_2\! =\! d }\\
            &=& \charf{  m_1 \! + \! 2m_2 \! = \! 0  }\charf{ d \! =\! \lambda_1 \! =\! \lambda_2 }\\
            &=& \charf{ \lambda \!=\! \mu \!=\! (d,d) }.
        \end{eqnarray*}

        Now, let $k=2d\!-\!1$, so $\nu=(1^{2d})$ is a one-column shape.  From Remark~\ref{rem:transposesymmetry}, we have that
          $$\kroncoeff_{\mu\, (1^{2d}) \, \lambda} = \kroncoeff_{\mu \lambda \, (1^{2d})}=\charf{ \lambda \!=\! \mu'\!=\!(2^d) }.$$  Also, the formula in part 2 is unchanged when calculating $\kroncoeff_{\mu\nu'\lambda'}$ for transposed $\nu $ and $\lambda$, so we use the above case when $k=0$ to show that
          $$\kroncoeff_{\mu\, (1^{2d}) \, \lambda} = \kroncoeff_{\mu\, (2d) \, \lambda'}=\charf{ \lambda' \!=\! \mu\!=\!(d,d) }=\charf{ \lambda \!=\! (2^d) }.$$

        Let $\lambda=(\lambda_1,\lambda_2,2^{m_2},1^{m_1})$ be a double hook.
        We consider two cases: $\lambda_1 \! - \! \lambda_2 \leq m_1$ and $\lambda_1 \! - \! \lambda_2 > m_1$.

        \begin{enumerate}
            \item[Case 1]  $\lambda_1 \! - \! \lambda_2 \leq m_1$: From Theorem \ref{Rosas1} part 3 with $\mu_1\! =\! \mu_2\! =\!d$ and $\nu_1\! =\!k$, we have that
                \begin{eqnarray*}
                   \kroncoeff_{\mu\nu\lambda} &=&  \charf{ \lambda_2 \leq d \! - \! m_2 \leq \lambda_1 \! }
                        \charf{ 0 \leq k \! - \! m_1 \! - \! 2m_2 \leq 3 \! }\\
                   && + \charf{ \lambda_2 \leq d \! - \! m_2 \! - \! 1\leq \lambda_1 \! }
                        \charf{ 1 \leq k \! - \! m_1 \! - \! 2m_2 \leq 2 \! }\\
                   && + \charf{ \lambda_2 \leq d \! - \! m_2 \! + \! 1 \leq \lambda_1 \! }
                        \charf{ 1 \leq k \! - \! m_1 \! - \! 2m_2 \leq 2 \! }\\
                   && - \charf{ \lambda_2 \! + \! m_2 \! + \! m_1 \! =\! d}
                        \charf{ 1 \leq k \! - \! m_1 \! - \! 2m_2 \leq 2 \! }.\\
                    & =& \charf{ 1 \! \leq \! k \! - \! m_1 \! - \! 2m_2 \! \leq \! 2 \! } \!
                    \left[ \!  - \charf{ \lambda_2 \! + \! m_2 \! + \! m_1 \! =\! d}\! +
                    \!\! \!  \sum_{-1\leq a\leq 1 }\!\!\! \charf{\lambda_2\! \leq \! d\! -\! m_2\! +\! a\! \leq \! \lambda_1\! } \! \right]
\\
          && +  \charf{ k \! - \! m_1 \! - \! 2m_2 \! = \! 0 \textrm{ or } 3 \! }
            \charf{ \lambda_2 \leq d \! - \! m_2 \leq \lambda_1 \! }. \!
        \end{eqnarray*}
Now, consider the following characteristic function
        $$\begin{array}{rll}
                \Big. \charf{ d\! +\! a \! \leq \! \lambda_1 \! + \! m_2 } \Big.
                  & = \charf{\lambda_2\! +\! m_2\! +\! m_1 \! \leq \! d\! -\! a }
                    & ( \textrm{as }\lambda\vdash2d)
                \\
                 & = \Big. \charf{\lambda_1\! +\! m_2\! \leq \! d\! -\! a } \Big.
                    & (\textrm{as }\lambda_1 \! - \! \lambda_2 \! \leq \! m_1)
                \\
                 & = \Big. \charf{\lambda_1\! +\! m_2\! \leq \! d\! -\! a } \Big.
                        \charf{ d\! +\! a \! \leq \! \lambda_1 \! + \! m_2 }
                \\
                 & = \Big.\charf{d \! +\! a \! \leq \! \lambda_1\! +\! m_2\! \leq \! d\! -\! a } \Big. .
        \end{array}$$
We use this to simplify the sum in the formula $\kroncoeff_{\mu\nu\lambda}$ above by using the identity
$$\charf{\lambda_2\! \leq \! d\! -\! m_2\! +\! a\! \leq \! \lambda_1\! } =
    \charf{\lambda_2 \! +\! m_2\! \! \leq \! d\!  +\! a }
            \charf{ d\! +\! a\! \leq \! \lambda_1\!  +\! m_2 }.$$
If $ a > 0$, then the inequality is impossible to satisfy, so $$\charf{d \! +\! a \! \leq \! \lambda_1\! +\! m_2\! \leq \! d\! -\! a }=0,$$ and so the $a\! =\! 1$ term in the sum is zero.
For the $ a=0$ term, since $\lambda\vdash 2d$ and $\lambda_2 \leq \lambda_1$ by the definition of partitions,
$$\charf{\lambda_2\! + \! m_2 \leq \! d }\charf{d\!\leq \!\lambda_1\!    +\! m_2\!\leq \! d}=
    \charf{\lambda_1\! +\! m_2\! =\! d }=\charf{ \lambda_2 \! + \! m_2 \! + \! m_1 \! =\! d}.$$

Applying these to $\kroncoeff_{\mu\nu\lambda}$, we get that
     \begin{eqnarray*}
        \kroncoeff_{\mu\nu\lambda} & =& \charf{ 1 \! \leq \! k \! - \! m_1 \! - \! 2m_2 \! \leq \! 2  } \!
            \Big[  \charf{ \lambda_2 \! + \! m_2 \! + \! m_1 \! =\! d}
            - \charf{ \lambda_2 \! + \! m_2 \! + \! m_1 \! =\! d } \! \Big]\\
        && +\Big. \charf{ 1 \! \leq \! k \! - \! m_1 \! - \! 2m_2 \! \leq \! 2  } \!
            \charf{d \! -\! 1 \! \leq \! \lambda_1\! +\! m_2\! \leq \! d\! +\! 1 }
            \charf{\lambda_2\! + \! m_2 \leq \! d\!-\! 1 } \Big.\\
          && +  \charf{ k \! - \! m_1 \! - \! 2m_2 \! = \! 0 \textrm{ or } 3  }
            \charf{\lambda_1\! +\! m_2\! =\! d }
                \charf{\lambda_2\! + \! m_2 \leq \! d } \\
        &=&  \charf{ 1 \! \leq \! k \! - \! m_1 \! - \! 2m_2 \! \leq \! 2  } \!
        \left( \begin{array}{l}
            \Big. \charf{d \! -\! 1 \! = \! \lambda_1\! +\! m_2 } \Big.
                \charf{\lambda_2\! + \! m_2 \leq \! d\!-\! 1 }\\
            + \Big. \charf{d \!  = \! \lambda_1\! +\! m_2 } \Big.
                \charf{\lambda_2\! + \! m_2 \leq \! d\!-\! 1 }\\
            + \Big. \charf{d \! +\! 1 \! = \! \lambda_1\! +\! m_2 } \Big.
                \charf{\lambda_2\! + \! m_2 \leq \! d\!-\! 1 }
        \end{array}\right) \\
          && +  \charf{ k \! - \! m_1 \! - \! 2m_2 \! = \! 0 \textrm{ or } 3 \! }
            \charf{\lambda_1\! +\! m_2\! =\! d }.
     \end{eqnarray*}
To simplify the large bracket, we have the following relations.
    $$\begin{array}{rll}
            \charf{d \! -\! 1 \! = \! \lambda_1\! +\! m_2 }
              & = \charf{d \! -\! 1 \! = \! \lambda_1\! +\! m_2 } \!
                \charf{\lambda_2\! + \! m_2 \leq \! d\!-\! 1 }. & ( \textrm{as } \lambda_2 \leq \lambda_1)
    \end{array}$$
    $$\begin{array}{rll}
            \charf{d \! +\! 1 \! = \! \lambda_1\! +\! m_2 }
              & = \Big. \charf{d \! -\! 1 \! = \! \lambda_2\! +\! m_2 \! + \! m_1 } \Big. & ( \textrm{as } \lambda \vdash 2d) \\
              & = \Big. \charf{d \! -\! 1 \! = \! \lambda_2\! +\! m_2 \! + \! m_1 }
                    \charf{\lambda_2\! + \! m_2 \leq \! d\!-\! 1 } \Big.  & ( \textrm{as } 0 \leq m_1) \\
              & = \Big. \charf{d \! +\! 1 \! = \! \lambda_1\! +\! m_2 }
                \charf{\lambda_2\! + \! m_2 \leq \! d\!-\! 1 }. \Big.
    \end{array}$$
    $$\begin{array}{rll}
            \charf{d \! = \! \lambda_1\! +\! m_2 } \! \charf{1 \! \leq \!  m_1 }
              & = \Big. \charf{d \!  = \! \lambda_2\! +\! m_2 \! + \! m_1 } \! \charf{1 \! \leq \!  m_1 } \\
              & =\Big. \charf{d \!  = \! \lambda_2\! +\! m_2 \! + \! m_1 } \!
                \charf{1 \! \leq \!  d \! - \! \lambda_2\! - \! m_2 } \Big. \\
              & = \Big. \charf{d \! = \! \lambda_1\! +\! m_2 } \!
                    \charf{\lambda_2\! + \! m_2 \! \leq \!  d \! - \! 1 } \Big..
    \end{array}$$
Thus,
\begin{eqnarray*}
    \kroncoeff_{\mu\nu\lambda} & =& \charf{ 1 \! \leq \! k \! - \! m_1 \! - \! 2m_2 \! \leq \! 2 \! } \!
        \Big[ \charf{d \!  = \! \lambda_1\! +\! m_2 }
        \charf{1 \! \leq \!  m_1 }
        + \charf{d \! \pm \! 1 \! = \! \lambda_1\! +\! m_2 } \Big] \\
        && +  \charf{ k \! - \! m_1 \! - \! 2m_2 \! = \! 0 \textrm{ or } 3 \! } \charf{\lambda_1\! +\! m_2\! =\! d }.
\end{eqnarray*}

\item[Case 2] $\lambda_1 \! - \! \lambda_2 > m_1$: Now let $\lambda=(\lambda_1,\lambda_2,2^{m_2},1^{m_1})$ be a double hook with $ m_1 < \lambda_1-\lambda_2 $, and $\nu=(2d \! -\! k,1^k)$.  Using the symmetry properties of the Kronecker coefficient, we have that $\kroncoeff_{\mu\nu\lambda}=\kroncoeff_{\mu\nu'\lambda'}$, where
    \begin{eqnarray*}
         && \nu'=(2d \! -\! k',1^{k'})=(k\! +\! 1,1^{2d - k - 1}),  \\
         && \lambda'=(\lambda_1',\lambda_2',2^{m_2'},1^{m_1'}) = (m_2\! +\! m_1\! +\! 2,m_2\! +\! 2,2^{\lambda_2 - 2},1^{\lambda_1 - \lambda_2}).
    \end{eqnarray*}
    $\kroncoeff_{\mu\nu'\lambda'}$ is given by the formula for Case 1, as $\nu'$ is still a hook, and
    $$\lambda_1'\!  - \! \lambda_2'\! = \! (m_2\! +\! m_1\! +\! 2) \! - \! (m_2\! +\! 2) \! =\! m_1  <  \lambda_1\! -\! \lambda_2 \! = \! m_1'.$$
    To use Case 1, we first simplify the following expressions using the fact $\lambda \vdash 2d$ repeatedly:
    \begin{eqnarray*}
        k' \! - \! m_1' \! - \! 2m_2' & =& (2d\! -\! k\! -\! 1) \! - \! (\lambda_1\! -\! \lambda_2) \! - \! 2(\lambda_2 \! - \! 2) \\
            &=& 2d \! - \! \lambda_1\! -\! \lambda_2\! -\! k \! +\! 3 \\
            &=& \Big. 2m_2 \! + \! m_1\! -\! k \! +\! 3 \Big. \\
       \Big.  \lambda_1' \! +\! m_2' &=& (m_2\! +\! m_1\! +\! 2) \! + \! (\lambda_2 \! - \! 2) \Big. \\
            &=& \lambda_2 \!+ \! m_2\! +\! m_1,
    \end{eqnarray*}
    and then use the above expressions in identities for each of the characteristic functions in $\kroncoeff_{\mu\nu'\lambda'}$:
    \begin{eqnarray*}
        \charf{ 1 \! \leq \! k' \! - \! m_1' \! - \! 2m_2' \! \leq \! 2 \! } &=&
            \charf{ 1\! - \! 3 \! \leq \! - \!(k \! - \! m_1 \! - \! 2m_2)  \! \leq \! 2\! -\! 3 \! } \\
          &= & \Big. \charf{ 1 \! \leq \! k \! - \! m_1 \! - \! 2m_2 \! \leq \! 2 \! },\Big. \\
        \Big. \charf{ k' \! - \! m_1' \! - \! 2m_2' \! = \! 0 \textrm{ or } 3 \! } &=&
            \charf{ -\! (k \! - \! m_1 \! - \! 2m_2) \! = \! -3 \textrm{ or } 0 \! } \\
          &=& \charf{ k \! - \! m_1 \! - \! 2m_2 \! = \! 0 \textrm{ or } 3 \! }, \Big.
    \end{eqnarray*}
    \begin{eqnarray*}
        \charf{d \! \pm \! 1 \! = \! \lambda_1' \! +\! m_2' } &=&
            \charf{d \! + \! 1 \! = \! \lambda_2 \!+ \! m_2\! +\! m_1 }
            + \charf{d \! - \! 1 \! = \! \lambda_2 \!+ \! m_2\! +\! m_1 } \\
          &=& \charf{d \! - \! 1 \! = \! \lambda_1 \!+ \! m_2 }
            + \charf{d \! + \! 1 \! = \! \lambda_1 \!+ \! m_2 }\\
          &=& \charf{d \! \pm \! 1 \! = \! \lambda_1 \! +\! m_2 },
    \end{eqnarray*}
    \begin{eqnarray*}
        \charf{d \!  = \! \lambda_1' \! +\! m_2' } \! \charf{1 \! \leq \!  m_1' }
          &=& \charf{d \! = \! \lambda_2 \!+ \! m_2\! +\! m_1 }\!
            \charf{1 \! \leq \! \lambda_1\! -\! \lambda_2 }\\
          &=& \charf{d \! = \! \lambda_1 \!+ \! m_2 } \!
            \charf{\lambda_1 \!+ \! m_2 \! = \! \lambda_2 \!+ \! m_2\! +\! m_1 } \!
            \charf{1 \! \leq \! \lambda_1\! -\! \lambda_2 }\\
          &=& \charf{d \! = \! \lambda_1 \!+ \! m_2 } \!
            \charf{\lambda_1 \! - \! \lambda_2 \!=\! m_1 } \!
            \charf{1 \! \leq \! m_1 }\\
          &=& \charf{d \! = \! \lambda_1 \!+ \! m_2 } \!
            \charf{1 \! \leq \! m_1 }.
    \end{eqnarray*}
Finally, we see that the above expressions for $\lambda'$ are identical to those for $\lambda$, so $\kroncoeff_{\mu\nu\lambda}$ is given by the formula from Case 1 for any $\lambda_1-\lambda_2$.
\end{enumerate}
        \item
            If $\lambda$ is a one-row shape or a one-column shape, then
            $\kroncoeff_{\mu\nu\lambda}=\charf{ \mu\!=\!\nu }$ or $\charf{ \mu\!=\!\nu' }$.  Now, $\nu$ and $\nu'$ are both hooks, so $2\leq d$, as hooks are partitions of $3$ or more.  Hence, $\mu=(d,d)$ has  $\mu_2 \geq 2$, so $\mu$ cannot be a hook, and $\charf{ \mu\!=\!\nu }= \charf{ \mu\!=\!\nu' }=0$.
            If $\lambda$ is not contained in a double hook, then from Theorem \ref{Rosas1}, $\kroncoeff_{\mu\nu\lambda}=0. $
\end{enumerate}
\end{proof}

\begin{corollary}\label{cor:multfree} The Kronecker product $s_{(d,d)} \ast s_{(2d-\! k,1^k)}$ is multiplicity free, that is, in
    $$s_{(d,d)}\ast s_{(2d-\! k,1^k)} = \sum_{ \lambda \vdash 2d} \kroncoeff_{\mu\nu\lambda} s_\lambda,$$
 we have    $\kroncoeff_{\mu\nu\lambda} = 1$ or $0$.
\end{corollary}

\begin{proof}This is clear from Proposition~\ref{ddcase} when $\lambda$ is a hook, a one-row shape, or a one-column shape.  If $\lambda$ is a double hook, then we note that the terms of the formula for double hooks in Proposition \ref{ddcase} are given by characteristic functions with mutually exclusive propositions, so no more than one term can equal $1$ for a given double hook $\lambda$. \end{proof}
}

\begin{theorem} \label{the:hookform}Let $d \geq 3$ and $0\leq k \leq 2d\!-\!1$.
%  Let $\lambda$, $\mu=(d,d)$, and $\nu=(2d\! -\! k,1^k)$ be partitions of $2d$, 
%and  let $\mu^+=(d\! +\! 1,d\! +\! 1) \vdash 2d\!+\!2$.  
%Some examples of cases that I don't think work with the orginal theorem as stated.
%s_22 * s_31 = s_211 + s_31
%terms from s_11 * s_2 = s_11
%+ other terms where d=1 and k=0
Define the map
\begin{eqnarray*}
    \Big. \phi: \{ \lambda \vdash 2d \} &\rightarrow& \{ \lambda \vdash 2(d\! +\! 1) \} \Big.  \\
   \lambda =(\lambda_1,\lambda_2,\ldots,\lambda_k) &\mapsto& (\lambda_1 \! +\! 1,\lambda_2,\ldots,\lambda_k,1).
\end{eqnarray*}
Then
\begin{align}
s_{(d+1,d+1)} \ast &s_{(2d-k+1,1^{k+1})} = 
s_{(d-h_{k-1}+1,d-h_{k-1},2^{h_{k-1}},1)}+s_{(d+1-r_k,d+1-r_k,2^{r_k})}\label{eq:hookform}\\
&+s_{(d-h_{k-2},d-h_{k-2},2^{h_{k-2}},1^2)}+s_{(d+2-h_{k},d-h_{k},2^{h_{k}})}
+ \sum_{ \lambda \vdash \, 2d} \kroncoeff_{(d,d)(2d-k,1^{k})\lambda}s_{\phi(\lambda)},\nonumber
\end{align}
with $h_{n}= \left\lfloor \frac{n}{2} \right\rfloor$ and $r_k= h_{k-1} + (k~mod~2)$.  On condition
that the Schur functions in  \eqref{eq:hookform} are not indexed by partitions, the terms
are assumed to be $0$.
\end{theorem}

%\begin{remark}\label{rem:xtraterms} If the value of $m_{k_i}$ or $m_e$ is such that $\zeta_a$, $
%\zeta_b$,$\zeta_c$, or $\zeta_d$ is not a partition of $2d\! +\!2$, then $s_{\zeta_a}$, $s_{\zeta_b}$, 
%$s_{\zeta_c}$, or $s_{\zeta_d}$ will not appear in $s_{\mu^+}\ast s_{\phi(\nu)}.$\end{remark}
\begin{proof}
Using Proposition \ref{Rosas1} (2), if $\lambda$ and $\gamma = \phi(\lambda)$ are hooks
we calculate that $\kroncoeff_{(d+1,d+1)(2d+1-k,1^{k+1})\gamma}- \kroncoeff_{(d,d)(2d-k,1^k)\lambda}
=0$.  We assume then that $\gamma = \phi(\lambda) =
(\lambda_1+1,\lambda_2,2^{a_2},1^{a_1})$ for
some $\lambda$ and calculate that
\begin{align*}
\kroncoeff_{(d+1,d+1)(2d+1-k,1^{k+1})\gamma}\! - \!\kroncoeff_{(d,d)(2d-k,1^k)\lambda} =
\charf{&\lambda_2\!-\!1\!=\!d\!-\!a_2\!\leq\!\lambda_1} 
\charf{a_1\!+\!2a_2\!\leq\!k\!\leq\!a_1\!+\!2a_2\!+\!3}\\+
\charf{&\lambda_2\!-\!1\!=\!d\!-\!a_2-1\!\leq\!\lambda_1} 
\charf{ a_1\!+\!2a_2+1\!\leq k\!\leq\!a_1\!+\!2a_2+2}\\+
\charf{&\lambda_2\!-\!1\!=\!d\!-\!a_2\!+\!1\!\leq\!\lambda_1} 
\charf{ a_1\!+\!2a_2\!+\!1\!\leq\!k\!\leq\! a_1\!+\!2a_2+2}~.
\end{align*}
This expression is equal to $1$ if and only if
$\gamma = (d-h_{k-1}+1,d-h_{k-1},2^{h_{k-1}},1)$ and it is equal to $0$ for all
other partitions $\gamma = \phi(\lambda)$.  This holds if $\lambda_1-\lambda_2 \leq a_1$
and $\kroncoeff_{(d+1,d+1)(2d+1-k,1^{k+1})\gamma}\! - \!
\kroncoeff_{(d,d)(2d-k,1^k)\lambda}=0$ otherwise (recall that
Proposition \ref{Rosas1} (3) holds only if $\lambda_1 - \lambda_2 \leq a_1$).

Now assume that $\gamma = (\gamma_1, \gamma_2, 2^{b_2},1^{b_1})$ 
is not in the image of $\phi$, then either $\gamma_1=\gamma_2$ or $b_1 = 0$.
If $\gamma_1 -\gamma_2 \leq b_1$, then we conclude that $\gamma_1 = \gamma_2$ and
we can apply Proposition \ref{Rosas1} (3).
We leave it to the reader to calculate directly $\kroncoeff_{(d+1,d+1)(2d+1-k,1^{k+1})\gamma}$
from Proposition \ref{Rosas1}
in this case and conclude that it is equal to $1$ (and $0$ otherwise) if and only if
$\gamma = (d-h_{k-2}, d - h_{k-2}, 2^{h_{k-2}}, 1^2)$ or $\gamma = (d+1-r_k,d+1-r_k, 2^{r_k})$.

Finally the remaining case to consider is when $\gamma$ is not in the image of $\phi$ and
$\gamma_1 -\gamma_2>b_1$.  We let
$\gamma' = (\alpha_1, \alpha_2, 2^{c_2}, 1^{c_1})$ where $c_1>0$ and $\alpha_1 = \alpha_2$
and again leave the detail of calculating $\kroncoeff_{(d+1,d+1)(k+2,1^{2d-k})\gamma'}$ to
the reader.  This coefficient is equal to $1$ (and $0$ otherwise) if and only if
$\gamma' = (h_k+2, h_k+2, 2^{d-2-h_k},1^2)$ and $\gamma = (d+2-h_k,d-h_k,2^{h_k})$.

These calculations show that
$$s_{(d+1,d+1)}\ast s_{(2d-k+1,1^{k+1})} \!- \!\sum_{ \lambda \vdash \, 2d} \kroncoeff_{(d,d)(2d-k,1^{k})\lambda}s_{\phi(\lambda)}$$
consists of precisely the $4$ terms stated in \eqref{eq:hookform} if they exist. 
\end{proof}

\comment{By iterating this recursive formula we arrive at the following direct
expression for the Kronecker product of $s_{(d,d)}$ with a Schur function
indexed by a hook.

\begin{corollary} \label{explicitforhooks}
For $d \geq 2$ and $0 \leq k < 2d$,
\begin{align*}
s_{(d,d)} \ast s_{(2d-k,1^k)} &= 
\sum_{r=0}^{min(k,2d-4-k)} s_{(d+m_{k-r}-h_{k-r},d-r-h_{k-r+1},2^{h_{k-r}},1^r)}\\
&~~~+ \sum_{r=0}^{min(k-2,2d-3-k)} s_{(d+1-h_{k-r}, d-r-h_{k-r}, 2^{h_{k-r-2}},1^{r+1})} \\
&~~~+ \sum_{r=0}^{min(k-2,2d-3-k)} s_{(d+2-h_{k-r+1}, d-r-1+m_{k-r} - h_{k-r-2}, 2^{h_{k-r-2}},1^r)}\\
&~~~+ \sum_{r=0}^{min(k-3,2d-2-k)} s_{(d-h_{k-r-1}, d-r-h_{k-r-1}, 2^{h_{k-r-3}},1^{r+2})}\\
&~~~+ \charf{d=k~or~k+1} s_{(d,1^d)}
+ \charf{d=k+1} s_{(d+1,1^{d-1})}\\
&~~~+ \charf{d\leq k \leq 2d-1} s_{(2d+1-k,2^{k-d},1^{2d-1-k})}
\end{align*}
where $h_k = \lfloor \frac{k}{2} \rfloor$ and $m_k = k~mod~2~.$
\end{corollary}

\begin{proof}
First we note that 
\begin{align*}
s_{(d-r_{k-1},d-r_{k-1},2^{r_{k-1}})}&+s_{(d+1-h_{k-1},d-1-h_{k-1},2^{h_{k-1}})}\\
&= s_{(d+m_{k}-h_{k},d-h_{k+1},2^{h_{k}})} + 
s_{(d+2-h_{k+1}, d-1+m_{k} - h_{k-2}, 2^{h_{k-2}})}~.
\end{align*}
This can easily be verified by using the definition of $r_k = h_{k-1} + m_k$ and
checking that for $k$ even and then for $k$ odd that this equality holds.

First the base case, when $k=0$ and $d \geq 2$ then the first sum has exactly $1$ term
$s_{(d,d)}$ and for $k=2d-1$ only the last term contributes to the expression and
it is $s_{(2^d)}$.  Both these are consistent with the formulas for $s_{(d,d)} \ast s_{(2d)}$ and
$s_{(d,d)} \ast s_{(1^{2d})}$.  If $d= 2$, the base cases may easily be checked 
separately for each $k=0,1,2,3$.
Now the proof proceeds by induction following the expressions from Theorem \ref{the:hookform}
and the expression for $s_{(d-1,d-1)} \ast s_{(2d-k-1,1^{k-1})}$.

\begin{align*}
s_{(d,d)} \ast s_{(2d-k,1^k)} = 
&s_{(d+m_{k}-h_{k},d-h_{k+1},2^{h_{k}})} 
+s_{(d-1-h_{k-2}+1,d-1-h_{k-2},2^{h_{k-2}},1)}\\
&+ s_{(d+2-h_{k+1}, d-1+m_{k} - h_{k-2}, 2^{h_{k-2}})}
+s_{(d-1-h_{k-3},d-1-h_{k-3},2^{h_{k-3}},1^2)}\\
&+ \sum_{ \lambda \vdash \, 2d} \kroncoeff_{(d,d)(2d-k,1^{k})\lambda}s_{\phi(\lambda)}\\
=&s_{(d+m_{k}-h_{k},d-h_{k+1},2^{h_{k}})} 
+s_{(d-1-h_{k-2}+1,d-1-h_{k-2},2^{h_{k-2}},1)}\\
&+ s_{(d+2-h_{k+1}, d-1+m_{k} - h_{k-2}, 2^{h_{k-2}})}
+s_{(d-1-h_{k-3},d-1-h_{k-3},2^{h_{k-3}},1^2)}\\
&+\sum_{r=1}^{min(k,2d-4-k)} s_{(d+m_{k-r}-h_{k-r},d-r-h_{k-r+1},2^{h_{k-r}},1^{r})}\\
&+ \sum_{r=1}^{min(k-2,2d-3-k)} s_{(d+1-h_{k-r}, d-r-h_{k-r}, 2^{h_{k-r-2}},1^{r+1})} \\
&+ \sum_{r=1}^{min(k-2,2d-3-k)} s_{(d+2-h_{k-r+1}, d-r-1+m_{k-r} - h_{k-r-2}, 2^{h_{k-r-2}},1^{r})}\\
&+ \sum_{r=1}^{min(k-3,2d-2-k)} s_{(d-h_{k-r-1}, d-r-h_{k-r-1}, 2^{h_{k-r-3}},1^{r+1})}\\
&+ \charf{d=k~or~k+1} s_{(d,1^{d})}
+ \charf{d=k+1} s_{(d+1,1^{d-1})}\\
&+ \charf{d\leq k \leq 2d-2} s_{(2d+1-k,2^{k-d},1^{2d-1-k})}
\end{align*}
The first four terms in this expression are the $r=0$ terms in the four sums
and expression holds by induction.
\end{proof}}

This theorem can be used along with an induction argument to
derive a non-recursive formula for the Kronecker product
$s_{(d,d)} \ast s_{(2d-k,1^k)}$.  We choose to not include this
result here for lack of an application of this formula.

\begin{subsection}{Stability of Kronecker coefficients}

We now use Theorem~\ref{the:hookform} to observe the precise stabilization of the Kronecker coefficients $\kroncoeff_{(d+k, d+k)(d+2k+1, 1^{d-1})\lambda}$.

\begin{corollary}
Let $d\geq 1$,  $\lambda = (\lambda _1, \lambda _2, \lambda_3, \ldots ,\lambda _\ell)\vdash 2d$ and $\tilde{\lambda } = (\lambda _1 +k, \lambda _2 +k, \lambda_3, \ldots ,\lambda _\ell)\vdash 2d$. Then for $k\geq 0$
$$\kroncoeff_{(d,d)(d+1, 1^{d-1})\lambda}=\kroncoeff_{(d+k, d+k)(d+2k+1, 1^{d-1})\tilde{\lambda}}.$$However, for $k<0$ there exists $\lambda$ for which
$$\kroncoeff_{(d,d)(d+1, 1^{d-1})\lambda}\neq\kroncoeff_{(d+k, d+k)(d+2k+1, 1^{d-1})\tilde{\lambda}}.$$
\end{corollary}

\begin{proof}
For $d=1$ we recall the well known result that for $k\geq 0$
\begin{equation}\label{eq:deq1}
s_{(k+1, k+1)}\ast s_{(2k+2)}= s_{(k+1, k+1)}.
\end{equation}
For $d=2$ we deduce immediately from \cite[Theorem 4.8]{BKron2} that for $k\geq 0$
\begin{equation}
\label{eq:deq2}
s_{(k+2, k+2)}\ast s_{(2k+3, 1)}= s_{(k+3, k+1)} + s_{(k+2, k+1, 1)}.
\end{equation}
For $d\geq 3$ repeated application of Theorem~\ref{the:hookform} to \eqref{eq:deq2} yields for $k\geq 0$
\begin{equation}\label{eq:deq3}\end{equation}
\begin{eqnarray*}
s_{(d+k, d+k)}\ast s_{(d+2k+1, 1^{d-1})}&=& s_{(d+1+k, 2+k, 1^{d-3})} + s_{(d+1+k, 1+k, 1^{d-2})}\\
&+& s_{(d+k, 2+k, 1^{d-2})} + s_{(d+k, 1+k, 1^{d-1})}\\
&+& s_{(d-1+k, 2+k, 2, 1^{d-3})}\\
&+&\sum _{m=0} ^{\lfloor \frac{d}{2} \rfloor -2} \left(s_{\phi ^{2-2m-2} (m+4+k, m+3+k, 2^m, 1)}\right.\\
&+& s_{\phi ^{2-2m-2}(m+4+k, m+4+k, 2^m)}\\
&+& s_{\phi ^{2-2m-2}(m+3+k, m+3+k, 2^m, 1^2) } \\
&+&\left. s_{\phi ^{2-2m-2} (m+4+k, m+2+k, 2^{m+ 1})}\right) \\
&+&\sum _{m=0} ^{\lfloor \frac{d}{2} \rfloor -2} \left( s_{\phi ^{2-2m-3}(m+4+k, m+3+k, 2^{m+1}, 1) }\right.\\
&+& s_{\phi ^{2-2m-3} (m+3+k, m+3+k, 2^{m+2})}\\
&+&  s_{\phi ^{2-2m-3}(m+4+k, m+4+k, 2^m, 1^2)}  \\
&+&  \left.s_{\phi ^{2-2m-3} (m+5+k, m+3+k, 2^{m+ 1})}\right) ~.
\end{eqnarray*}
For $k<0$ note that $\kroncoeff_{(-k, -k)(-2k)(-2k)}=0$ by \eqref{eq:deq1} and hence $\kroncoeff_{(-k, -k)(1^{-2k})(1^{-2k})}=0$ by the symmetry of Kronecker coefficients. Hence, by applying Theorem~\ref{the:hookform} $d+2k$ times we have $\kroncoeff_{(d+k, d+k)(d+2k+1, 1^{d-1})(d+2k+1, 1^{d-1})}=0$. However, from \eqref{eq:deq1}, \eqref{eq:deq2}, \eqref{eq:deq3} we see $\kroncoeff_{(d, d)(d+1, 1^{d-1})(d+1, 1^{d-1})}=1$, and the result follows.
\end{proof}
\end{subsection}

\end{section}

\comment{
\begin{proof}The image of $\phi$ is the set Im$\phi= \{ \phi(\lambda) | \lambda \vdash 2d \}$, and we can write
$$s_{(d+1,d+1)}\ast s_{(2d+1-k,1^{k+1})}= \sum_{ \zeta \in \textrm{Im}\phi} \kroncoeff_{(d+1,d+1) \phi(\nu) \zeta} \; s_\zeta + \sum_{ \zeta \not{\in} \textrm{Im}\phi} \kroncoeff_{(d+1,d+1) \phi(\nu) \zeta} \; s_\zeta.$$
We treat both sums separately, starting with the first.

Let $\zeta \in \textrm{Im}\phi$.  Then there is a partition $\lambda \vdash 2d$ such that $\phi(\lambda) = \zeta = (\lambda_1 \! +\! 1, \lambda_2,\ldots,\lambda_l,1)$. We have three cases to consider.  First, when $\lambda=(\lambda_1,1^m)$, $\phi(\lambda)=(\lambda_1\! +\! 1,1^{m + 1})$, so Proposition \ref{ddcase} gives that
\begin{eqnarray*}
    \kroncoeff_{(d+1,d+1) (2d+1-k,1^{k+1}) (\lambda_1\! +\! 1,1^{m + 1})}
    & =&\charf{ d\! +\! 1 \! - \! 1 \! \leq \! m \! +\! 1 \! \leq d\! + \! 1 }
                \charf{ d\! + \! 1 \! - \! 1 \! \leq \! k\! + \! 1 \! \leq d\! + \! 1 }\\
        &=& \Big. \charf{ d\! - \! 1 \! \leq \! m\leq d} \Big.
                \charf{ d \! - \! 1 \! \leq \! k \! \leq d } \\
        &=&  \kroncoeff_{(d,d) (2d-k,1^{k}) (\lambda_1,1^{m})}.
\end{eqnarray*}

Second, when $\lambda=(\lambda_1,\lambda_2,2^{m_2},1^{m_1})$ is a double hook, Proposition \ref{ddcase} gives that
\begin{eqnarray*}
    &&\kroncoeff_{(d+1,d+1) (2d+1-k,1^{k+1}) (\lambda_1+1,\lambda_2,2^{m_2},1^{m_1+1})}\\& =&
        \charf{ 1 \! \leq \! k \! + \! 1 \! - \! m_1 \! - \! 1 \! - \! 2m_2 \! \leq \! 2  } \!
        \left[ \begin{array}{l}
            \charf{d \! + \! 1 \!  = \! \lambda_1 \! + \! 1 \! +\! m_2 }
            \charf{1 \! \leq \!  m_1\! + \! 1 }  \\
            \Big. + \charf{d\! + \! 1 \! \pm \! 1 \! = \! \lambda_1\! + \! 1 \! +\! m_2 } \Big.
        \end{array} \right] \\
        && +  \charf{ k \! + \! 1 \! - \! m_1\! - \! 1 \! - \! 2m_2 \! = \! 0 \textrm{ or } 3  }
            \charf{d \! + \! 1 =\lambda_1\! + \! 1 \! +\! m_2 } \\
    &=& \charf{ 1 \! \leq \! k \! - \! m_1 \! - \! 2m_2 \! \leq \! 2  } \!
        \Big[ \charf{d \!  = \! \lambda_1\! +\! m_2 }
        + \charf{d \! \pm \! 1 \! = \! \lambda_1\! +\! m_2 } \Big] \\
        && +  \charf{ k \! - \! m_1 \! - \! 2m_2 \! = \! 0 \textrm{ or } 3  }
        \charf{d\! +\! 1 \! =\! \lambda_1 \! +\! 1\! +\! m_2 } \\
    &=& \Big. \kroncoeff_{(d,d)(2d-k,1^{k})\lambda} + \charf{ 1 \! \leq \! k  \! - \! 2m_2 \! \leq \! 2  } \!
        \charf{d   = \! \lambda_1 \! +\! m_2 } \! \charf{m_1\! = \! 0 } \Big. \\
    &=& \Big. \kroncoeff_{(d,d)(2d-k,1^{k})\lambda} + \charf{ m_2 \! = \! m_{k-1}  } \!
        \charf{\lambda_1 \! =\! d \! - \! m_{k-1} } \! \charf{m_1\! = \! 0 } \Big.\\
    &=& \Big. \kroncoeff_{(d,d)(2d-k,1^{k})(\lambda_1,\lambda_2,2^{m_2},1^{m_1})} + \charf{ \phi((\lambda_1,\lambda_2,2^{m_2},1^{m_1}))=\zeta_a } \Big..
\end{eqnarray*}
Finally, when $\lambda$ is not contained in a double hook, $\phi(\lambda)$ is also not contained in a double hook, so by Proposition \ref{ddcase},
$$ \kroncoeff_{(d+1,d+1) (2d+1-k,1^{k+1}) \phi(\lambda)} = 0 = 
\kroncoeff_{(d,d)(2d-k,1^{k})\lambda}.$$
Putting these three cases together, we have that
$$\sum_{ \zeta \in \textrm{Im}\phi} \kroncoeff_{(d+1,d+1) (2d+1-k,1^{k+1}) \zeta} \; s_\zeta = \, s_{\zeta_a} + \sum_{ \lambda \vdash 2d} \kroncoeff_{(d,d) (2d-k,1^{k}) \lambda} \; s_{\phi(\lambda)}.$$

We now examine the case when $\zeta \vdash 2d\!+\!2$ is not in the image of $\phi$; we need only consider $\zeta$ when it is a hook or double hook, as these are the only partitions for which $\kroncoeff_{(d+1,d+1) (2d+1-k,1^{k+1}) \, \zeta}$ might be non-zero.  If $\zeta=(\zeta_1, 1^{n_1})$ is a hook and $ \zeta \notin \textrm{Im} \phi$, then $(\zeta_1\!-\!1, 1^{n_1 -1})$ is not a partition, that is, $\zeta_1 =1$ or $n_1=0$.  Hence, $\zeta =(1^{2d+2})$ or $(2d\!+\!2)$, and from Proposition \ref{ddcase} part 3, we see that $\kroncoeff_{(d+1,d+1) (2d+1-k,1^{k+1}) \, \zeta}  = 0$ in either case for all $d$.

If $\zeta = (\zeta_1, \zeta_2, 2^{n_2}, 1^{n_1})$ is a double hook and $\zeta \notin \textrm{Im}\phi$, then $\zeta=(\zeta_1 \! -\! 1, \zeta_2, 2^{n_2}, 1^{n_1 - 1 })$ cannot be a partition, that is, $\zeta_1=\zeta_2$ or $n_1\! =\! 0$.  If $\zeta_1=\zeta_2$, then
$$\charf{d \! +\! 1 \! +\! a\! = \! \zeta_1\! +\! n_2 }=
    \charf{ \zeta_1\! +\! n_2 -\!a = d \! +\! 1  =  \zeta_2\! +\! n_2 \! +\! n_1 \! +\! a}=
    \charf{d \! +\! 1 \! +\! a\! = \! \zeta_1\! +\! n_2 }\charf{ -\!\!2a\! = \! n_1 }.$$
Combining the formula for $\kroncoeff_{(d+1,d+1) (2d+1-k,1^{k+1}) \, \zeta}$ from Proposition \ref{ddcase}, and the above, we get that
    \begin{eqnarray*}
        &&\kroncoeff_{(d+1,d+1) (2d+1-k,1^{k+1}) \, \zeta}\\
         & =& \charf{ 1 \! \leq \! k\! + \! 1\!- \! n_1 \! - \! 2n_2 \! \leq \! 2 \! } \!
            \Big[ \charf{d \! +\! 1 = \! \zeta_1\! +\! n_2 }\charf{ 0 \! = \! n_1 }
            \charf{1 \! \leq \!  n_1 }
            + \charf{d \! +\! 1 \! \pm \! 1 \! = \! \zeta_1\! +\! n_2 }\charf{ \mp\!2\! = \! n_1 } \Big] \\
            && +  \Big. \charf{ k \! + \! 1\! - \! n_1 \! - \! 2n_2 \! = \! 0 \textrm{ or } 3 \! }
            \charf{d\! +\! 1 \!=\! \zeta_1\! +\! n_2 }\charf{ 0 \! = \! n_1 } \Big.\\
        &=& \Big. \charf{ 1 \! \leq \! k\! - \! 1 - \! 2n_2 \! \leq \! 2 \! }
            \charf{d \! = \! \zeta_1\! +\! n_2 }\charf{ 2\! = \! n_1 } \Big.
            +  \charf{ k \! + \! 1 \! - \! 2n_2 \! = \! 0 \textrm{ or } 3 \! }
            \charf{d\! +\! 1 \!=\! \zeta_1\! +\! n_2 }\charf{ 0 \! = \! n_1 }\\
        &=& \Big. \charf{ n_2 \! = \! m_{k-2} \! }
            \charf{ \zeta_1\! = \! d  \!-\! m_{k-2} }\charf{ n_1 \! = \! 2} \Big.
           +  \charf{ n_2 \! =  m_e }
            \charf{\zeta_1  \!=\! d\! +\! 1 \! -\! m_e }\charf{n_1 \! =  0 }\\
        &=& \Big. \charf{ \zeta = \zeta_c } \Big. + \charf{ \zeta = \zeta_b }.
    \end{eqnarray*}
If $n_1\! =\! 0$, then
$$ \Big. \charf{d \! +\! 1 \! +\! a\! = \! \zeta_1\! +\! n_2 } \Big.=
    \charf{ \zeta_1\! +\! n_2 -\!a = d \! +\! 1  =  \zeta_2\! +\! n_2 \! +\! a}
    = \Big. \charf{d \! +\! 1 \! +\! a\! = \! \zeta_1\! +\! n_2 }
        \charf{ \zeta_1\! =  \zeta_2\! + 2a} \Big..$$
Combining the above with Proposition \ref{ddcase} gives
    \begin{eqnarray*}
       && \kroncoeff_{(d+1,d+1) (2d+1-k,1^{k+1}) \, \zeta}\\ & =& \charf{ 1 \! \leq \! k\! + \! 1\! - \! 2n_2 \! \leq \! 2 \! } \!
             \charf{d \! +\! 1 \! \pm \! 1 \! = \! \zeta_1\! +\! n_2 }
             \charf{ \zeta_1\! =  \zeta_2\! \pm 2} \\
            && +  \Big. \charf{ k \! + \! 1\! - \! 2n_2 \! = \! 0 \textrm{ or } 3 \! }
            \charf{d\! +\! 1 \!=\! \zeta_1\! +\! n_2 }\charf{ \zeta_1\! =  \zeta_2} \Big.
            \\
        &=& \Big. \charf{ 1 \! \leq \! k\! + \! 1 - \! 2n_2 \! \leq \! 2 \! }
            \charf{d \!+\! 2\! = \! \zeta_1\! +\! n_2 }\charf{ \zeta_1\! =  \zeta_2\! + \! 2} \Big.\\
         &&   +  \charf{ k \! + \! 1 \! - \! 2n_2 \! = \! 0 \textrm{ or } 3 \! }
            \charf{d\! +\! 1 \!=\! \zeta_1\! +\! n_2 }\charf{ \zeta_1\! =  \zeta_2}\\
        &=& \Big. \charf{ n_2 \! = \! m_k \! }
            \charf{ \zeta_1\! = \! d  \!+\! 2 \! -\! m_k }\charf{ \zeta_1\! =  \zeta_2\! + \! 2}
           +  \charf{ n_2 \! =  m_e }            \Big.
            \charf{\zeta_1  \!=\! d\! +\! 1 \! -\! m_e }\charf{ \zeta_1\! =  \zeta_2}\\
        &=& \Big. \charf{ \zeta = \zeta_d } \Big. + \charf{ \zeta = \zeta_b }.
    \end{eqnarray*}
These two formulas for $\kroncoeff_{(d+1,d+1) (2d+1-k,1^{k+1}) \, \zeta}$ show that
$$\sum_{ \zeta \not{\in} \textrm{Im}\phi} \kroncoeff_{(d+1,d+1) (2d+1-k,1^{k+1}) \zeta} \; s_\zeta = s_{\zeta_b}+s_{\zeta_c}+s_{\zeta_d},$$
which, when combined with the sum for $\textrm{Im}\phi$, completes the proof. \end{proof}

%%%%%END OF HOOKS%%%%%%%
}

\comment{   %remove this line to put back other version of sdd*sd+kd-k
\begin{section}{The Kronecker product $s_{(d,d)}\ast s_{(d+ k,d- k)}$} \label{two row case}

In this section, we present a theorem that fully describes the product $s_{(d,d)}\ast s_{(d+\! k,d-\! k)}$ in terms of what we have called rugs.  We begin with the definition of a rug.

\begin{definition}\label{schur set def}
    Let $X$ be a set of partitions of $2d$, then we define
    $$\rug{X}:=\sum_{\lambda \in X}s_\lambda$$
    to be a rug, and call $X$ the  index set of the rug $\rug{X}$.
If $X$ and $Y$ are disjoint sets of partitions, then $$\rug{X\uplus Y}=\rug{X} +\rug{Y},$$
where $\uplus$ denotes  disjoint union. \end{definition}

In the Kronecker product of two two-row shapes, the coefficient $\kroncoeff_{\mu\nu\lambda}$ is zero if $\lambda$ has more than 4 parts by \cite[Theorem 1.6]{Dvir}, so we take $$ \lambda = (\lambda_1,\lambda_2,\lambda_3, \lambda_4) $$ where some parts may be zero, say $\lambda$ is a \emph{four part partition} and restrict our attention to such partitions from now on.

We say that the differences $\lambda_1-\lambda_2$, $\lambda_2-\lambda_3$, and $\lambda_3-\lambda_4$ are \emph{gaps}, and the differences $ \lambda_1-\lambda_3$, and $\lambda_2-\lambda_4 $ are \emph{double gaps}.   For each four part partitions $\lambda$, we define $g_\lambda$ to be the the size of the smallest gap and $d_\lambda$ the size of the smallest double gap, that is,
$$\begin{array}{l}
    \Big. g_\lambda  =min\{\lambda_1-\lambda_2, \lambda_2-\lambda_3, \lambda_3-\lambda_4 \},  \Big. \\
     d_\lambda = min \{ \lambda_1-\lambda_3, \lambda_2-\lambda_4 \}.
\end{array}$$

Let the \emph{gap set}, $G_i$, be the set of four part partitions such that all their gaps are at least  $i$, and the \emph{double gap set}, $D_i$, be the set where all their double gaps are at least $i$, that is,
\begin{eqnarray*}
    G_i &=& \Big. \{ \lambda : i \leq g_\lambda \} \Big. \\
    D_i &=& \{ \lambda : i \leq d_\lambda \}.
\end{eqnarray*}
It is easy to see that $G_i \subset D_i$, and if $i \leq j$, then $G_j \subseteq G_i$ and $D_j \subseteq D_i$.

Define  $P$ to be the set of partitions that satisfy  $\lambda_1-\lambda_2$, $\lambda_2-\lambda_3$, and $\lambda_3-\lambda_4$ are all even integers.  This requirement means that all the parts of all partitions in $P$ have the same parity, odd or even:
$$P=\{\lambda \ | \ \lambda _1\equiv \lambda _2\equiv \lambda _3 \equiv \lambda _4 \mod 2\}.$$  In the set of all four part partitions of $2d$, $\overline{P}$ is the complement of $P$, and is the set of four part partitions that have at least one odd gap, and so its partitions have two odd and two even parts as all the parts must add to an even number.

Our results also require the definition of three other sets:
    \begin{itemize}
        \item $ X_k = \{ \lambda \textrm{ has all odd gaps} \} \cap
                \{ \lambda \; | \; \lambda_1\!-\!\lambda_2 < k,  \textrm{ or } \lambda_3\!-\!\lambda_4 < k \} \cap
                \{ \lambda \; | \; \lambda_2\!-\!\lambda_3 < k \} $
        \item $ Y_k = \{ \lambda \; | \; \lambda_2\!-\!\lambda_3 \textrm{ even and other gaps odd} \} \cap
                    \{ \lambda \; | \; \lambda_2\!-\!\lambda_3 < k \}$
        \item $Z_k = \{ \lambda \; | \; \lambda_2\!-\!\lambda_3 \textrm{ odd and other gaps even} \} \cap
                    \{ \lambda \; | \; \lambda_1\!-\!\lambda_2<k,  \textrm{ or } \lambda_3\!-\!\lambda_4 < k \}$.
    \end{itemize}

We are now in a position to state the decomposition of $s_{(d,d)}\ast s_{(d+\! k,d-\! k)}$ into rugs.

}\comment{  %remove this line to put back other version of sdd*sd+kd-k

\begin{theorem} \label{main theo}  Let $\mu=(d,d)$ and $\nu = (d \! + \! k,d \! - \! k)$.  Then
\begin{eqnarray*}
        s_{\mu}\ast s_{\nu} & = & \sum_{ \stackrel{1 \leq \, a \, \leq k-1  }{ a \, \textrm{ odd} }} \rug{ G_a \! \cap \! D_{a+k} } +
        \left\{  \begin{array}{ll}
            \rug{  D_k \! \cap \! P } + \rug{  D_k \! \cap \! X_k } & \textrm{if k is even} \\ \\
            \rug{  G_k \! \cap \! \overline{P} } + \rug{  D_k \! \cap \! \big( Y_k \uplus Z_k \big) } & \textrm{if k is odd.}
        \end{array} \right.
    \end{eqnarray*}
\end{theorem}

The proof of Theorem \ref{main theo} relies heavily on the following proposition, whose proof we postpone until the next section.

    \begin{proposition} \label{form prop}
        Let $ \mu = (d,d)$, $\nu = (d \! + \! k,d \! - \! k)$, and $\lambda = (\lambda_1, \lambda_2, \lambda_3, \lambda_4)$ all be partitions of $2d$.  Let $M= min\{\lambda_1  \! -  \! \lambda_2, \lambda_3  \! -  \! \lambda_4\}$. Then
   \begin{equation}  \label{Formula}  \begin{array}{rcl} 
                \kroncoeff_{\mu\nu\lambda} & = &
                \sum_{ \stackrel{1 \leq \, a \, \leq k-1  }{ a \; odd}}
                \charf{ a  \! \leq g_\lambda  }
                \charf{ a  \! +  \! k \leq d_\lambda }
                \\  & & +  \charf{ k \textrm{ even}   }
                \charf{ k  \! \leq  \! d_\lambda }
                \left( \begin{array}{l}
                    \Big. \charf{ M  \! <  \! k }
                    \charf{ \lambda_2  \! -  \! \lambda_3  \! <  \! k }
                    \charf{ M \textrm{ odd}  }
                    \charf{ \lambda_2  \! -  \! \lambda_3 \textrm{ odd}   } \Big. \\
                  + \Big. \charf{ M \textrm{ even}   }
                    \charf{ \lambda_2  \! -  \! \lambda_3 \textrm{ even}   } \Big.
                \end {array} \right)
                \\
                \\ & & +  \charf{ k \textrm{ odd}   }
                \charf{ k  \! \leq  \! d_\lambda }
                \left( \begin{array}{l}
                    \Big. \charf{ \lambda_2  \! -  \! \lambda_3  \! <  \! k }
                    \charf{ M \textrm{ odd}  }
                    \charf{ \lambda_2  \! -  \! \lambda_3 \textrm{ even}   } \Big.
\\
                     + \Big. \charf{ M  \! <  \! k }
                    \charf{ M \textrm{ even}  }
                    \charf{ \lambda_2  \! -  \! \lambda_3 \textrm{ odd}  } \Big.
\\
                    + \Big. \charf{ k  \! \leq  \! g_\lambda }
                    \charf{ M \textrm{ odd or } \lambda_2  \! -  \! \lambda_3 \textrm{ odd}  } \Big.
                \end{array} \right).\end{array}
        \end{equation}
    \end{proposition}

We now proceed with the proof of Theorem~\ref{main theo}.

\begin{proof}
The Kronecker product for $ \mu = (d,d)$, and $\nu = (d \! + \! k,d \! - \! k)$ is given by
$$s_{\mu}\ast s_{\nu} = \sum_{ \lambda \vdash 2d} \kroncoeff_{\mu\nu\lambda} s_\lambda, $$
so the result will be established if we can show that the coefficient of $s_\lambda$ in the rug decomposition given in Theorem \ref{main theo} is equal to $\kroncoeff_{\mu\nu\lambda}$ given by \eqref{Formula}.

We first note the following equation that the gaps of a four part partition $\lambda$ must satisfy
$$ 2d= 4\lambda_4 +3(\lambda_3\!-\!\lambda_4)+2(\lambda_2\!-\!\lambda_3)+(\lambda_1\!-\!\lambda_2).$$
Since the left side is even and $\lambda_3\!-\!\lambda_4$ and $\lambda_1\!-\!\lambda_2$ have odd coefficients, $\lambda _1-\lambda _2$ and $\lambda _3-\lambda _4$ must have the same parity, so if $M=min\{\lambda_1  \! -  \! \lambda_2, \lambda_3  \! -  \! \lambda_4\}$ is even, then $\lambda _1-\lambda _2$ and $\lambda _3-\lambda _4$ are both even, and similarly if $M$ is odd.
\end{proof}}\comment{\begin{proof}(continued) 
Now, let $b$ be a non-negative integer. From the notes and set definitions proceeding Theorem \ref{main theo}, we have     \begin{eqnarray*}
        \charf{ \lambda \in G_b } & =& \Big. \charf{ b  \! \leq g_\lambda  }, \Big.
\\
        \charf{ \lambda \in D_{b} } & = &
            \Big. \charf{ b  \leq d_\lambda }, \Big.
\\
        \charf{ \lambda \in P } & = & \Big. \charf{ M \textrm{ even}   }
                    \charf{ \lambda_2  \! -  \! \lambda_3 \textrm{ even}   }, \Big.
\\
        \charf{ \lambda \in \overline{P} } & = & \Big.
            \charf{ M \textrm{ odd or } \lambda_2  \! -  \! \lambda_3 \textrm{ odd}  }, \Big.
\\
        \charf{ \lambda \in X_k } & = & \Big. \charf{ M  \! <  \! k }
                    \charf{ \lambda_2  \! -  \! \lambda_3  \! <  \! k }
                    \charf{ M \textrm{ odd}  }
                    \charf{ \lambda_2  \! -  \! \lambda_3 \textrm{ odd}   }, \Big.
\\
        \charf{ \lambda \in Y_k } & = & \Big. \charf{ \lambda_2  \! -  \! \lambda_3  \! <  \! k }
                    \charf{ M \textrm{ odd}  }
                    \charf{ \lambda_2  \! -  \! \lambda_3 \textrm{ even}   }, \Big.
\\
        \charf{ \lambda \in Z_k } & = & \Big. \charf{ M  \! <  \! k }
                    \charf{ M \textrm{ even}  }
                    \charf{ \lambda_2  \! -  \! \lambda_3 \textrm{ odd}  }. \Big.
        \end{eqnarray*}

A set intersection is equivalent to multiplying propositional functions, and a disjoint union corresponds to adding them, so every rug appearing in Theorem~\ref{main theo} corresponds to a unique term of \eqref{Formula}.  Since $\lambda$ satisfying a term of \eqref{Formula} and $s_\lambda$ appearing in a rug both add one to the coefficient of $s_\lambda$ in $s_{\mu}\ast s_{\nu}$, a Schur function's coefficient in the rug decomposition of the theorem is equal to $\kroncoeff_{\mu\nu\lambda}$ given by \eqref{Formula}, establishing the result.\end{proof}

\begin{corollary}\label{max coeff}
The maximum coefficient of a Schur function in $s_{(d,d)}\ast s_{(d+k,d-k)}$ is $\left\lfloor \frac{k}{2}\right\rfloor\!+\! 1 $, and it occurs at Schur functions indexed by partitions in the set
$$\left. \begin{array}{ll}
            \Big. P \! \cap \! G_{k-1}  & \textrm{if k is even,} \Big. \\
            \Big. \overline{P} \! \cap \! G_{k-1}  & \textrm{if k is odd.} \Big.
    \end{array}\right. $$
\end{corollary}

\begin{proof} From  Theorem~\ref{main theo}, we see that a Schur function with maximum coefficient must appear in all the rugs of the sum, and one of the rugs that depends on $k$ being even or odd.  There are $\left\lfloor \frac{k}{2}\right\rfloor $ rugs in the sum in the formula, so $\left\lfloor \frac{k}{2}\right\rfloor +1 $ is the maximum coefficient.

Now, if $c \leq b$, then $G_b \subseteq G_c$ and $D_{b+k} \subseteq D_{c+k}$, so for all $1\leq a\leq k-1$,
\begin{eqnarray*}
    \big. G_{k-1} \! \cap \! D_{2k-1} \subseteq G_a \! \cap \! D_{a+k} \big. && \textrm{if $k$ is even, }\\
    \big. G_{k-2} \! \cap \! D_{2k-2} \subseteq G_a \! \cap \! D_{a+k} \big. && \textrm{if $k$ is odd.}
\end{eqnarray*}
So if $\lambda \in G_{k-1} \! \cap \! D_{2k-1}$ and $k$ is even, or $G_{k-2} \! \cap \! D_{2k-2}$ and $k$ is odd, then $s_\lambda$ appears in all the rugs of the sum.  To have the highest coefficient, $s_\lambda$ must also appear in one of the remaining rugs.
\end{proof}}\comment{\begin{proof}(continued)   %remove this line to put back other version of sdd*sd+kd-k
Take $k$ to be even, and let $\lambda \in G_{k-1} \! \cap \! D_{2k-1}$.  For $s_\lambda$ to have the highest coefficient, $\lambda$ must be in one of the remaining disjoint index sets.  Suppose first that $\lambda \in D_k \! \cap \! X_k$.  Then $\lambda _2-\lambda _3$ and one of $\lambda _1-\lambda _2$ or $\lambda _3-\lambda _4$ is strictly less than $k$, and all the gaps of $\lambda$ are greater than or equal to $k-1$. Hence, $\lambda _2-\lambda _3$ and one of $\lambda _1-\lambda _2$ or $\lambda _3-\lambda _4$ are equal to $k-1$.  Thus, one of $\lambda _1-\lambda _3$ or $\lambda _2-\lambda _4$ is $2k-2$, and so
$$D_{2k-1} \! \cap \! G_{k-1} \! \cap \! D_k \! \cap X_k =\emptyset. $$ Now, suppose $\lambda \in D_k \! \cap \! P$. Then
\begin{eqnarray*}
    \lambda \in G_{k-1} \! \cap \! D_{2k-1}\! \cap \!D_k \! \cap \!  P &=& G_{k-1} \! \cap \! D_{2k-1}\! \cap \!  P\\
        &=& G_{k} \! \cap \! D_{2k-1}\! \cap \!  P \quad \quad (\textrm{as all gaps must be even})\\
        &=& G_{k} \! \cap \!  P \quad \quad \quad \quad (\textrm{as all double gaps at least 2k})\\
        &=& G_{k-1} \! \cap \! P \quad \quad \quad \quad (\textrm{as $k-1$ is odd}),
\end{eqnarray*}
as required when $k$ is even.

Take $k$ to be odd.  We intersect $G_{k-2} \! \cap \! D_{2k-2}$ with $ G_k \! \cap \! \overline{P}$  and each disjoint set of the remaining rug.  For the first set, the intersection gives
$$ G_{k-2} \! \cap \! D_{2k-2} \! \cap \! G_k \! \cap \!  \overline{P}=G_k \! \cap \!  \overline{P}.$$
For the second, we have that if
$$\lambda \in G_{k-2} \! \cap \! D_{2k-2} \! \cap \! D_k \! \cap \!  Y_k
= G_{k-2} \! \cap \! D_{2k-2}\! \cap \!  Y_k,$$
then by the definition of $Y_k$ we have $\lambda _2-\lambda _3$ is even and obeys $k\! -\! 2 \leq \lambda_2\! -\! \lambda_3 < k$, so $k\! -\! 1 = \lambda_2\! -\! \lambda_3$.  The double gaps are greater than or equal to $2k-2$, and $k\! -\! 1 = \lambda_2\! -\! \lambda_3$, so $\lambda _1-\lambda _2$ and $\lambda _3-\lambda _4$ are greater than or equal to $k-1$. Hence, we have that
\begin{equation*} \label{sec int}
     G_{k-2} \! \cap \! D_{2k-2}\! \cap \!  Y_k
    = G_{k-1} \! \cap \! Y_k  .
\end{equation*}
Since $k-1$ is even, if $\lambda \in G_{k-1}  \! \cap \! \{ \lambda \; | \; \lambda_2\!-\!\lambda_3 < k \}$, then $\lambda_2\! -\! \lambda_3 \! = \! k\! -\! 1$ is even, so we have that $$ G_{k-1} \! \cap \! Y_k = G_{k-1}  \! \cap \! \overline{P} \! \cap\! \{ \lambda \; | \; \lambda_2\!-\!\lambda_3 < k \}. $$ Using a similar argument for $G_{k-2} \! \cap \! D_{2k-2} \! \cap \! D_k \! \cap \! Z_k$, we have that this set equals
$$ G_{k-1} \! \cap \! Z_k = G_{k-1} \! \cap \!  \overline{P}  \cap   \{ \lambda \; | \; \lambda_1\!-\!\lambda_2 < k,  \textrm{ or } \lambda_3\!-\!\lambda_4 < k \} .$$

We note that $\{ \lambda \; | \; \lambda_1\!-\!\lambda_2 < k,  \textrm{ or } \lambda_3\!-\!\lambda_4 < k \} \cup \{ \lambda \; | \; \lambda_2\!-\!\lambda_3 < k \}= \overline{G}_k$, the complement of $G_k$, so the union the $ (G_{k-1}  \cap  Y_k) \cup (G_{k-1}  \cap   Z_k )= G_{k-1} \! \cap \! \overline{P} \! \cap \!\overline{G}_k$. Finally, we have that the union of all three sets is
\begin{eqnarray*}
        (G_{k}  \cap  \overline{P}) \cup (G_{k-1}  \cap  Y_k) \cup (G_{k-1}  \cap   Z_k ) &=&
\Big. (G_{k}  \cap  \overline{P}) \cup (G_{k-1} \! \cap \! \overline{P} \! \cap \!\overline{G}_k) \Big.  \\
    &=& \Big. \overline{P} \cap \left(G_k \cup \left( G_{k-1}  \! \cap \! \overline{G}_k\right) \right)\Big. \\
    &=& \Big. \overline{P} \cap \left(\left( G_{k-1} \cup G_k \right)  \! \cap \! \left(\overline{G}_k\right) \cup G_k \right)\Big. \\
    &=& \overline{P} \cap  G_{k-1},
\end{eqnarray*}
as required. \end{proof}

}\comment{%remove this line to put back other version of sdd*sd+kd-k

\begin{corollary}%\label{cor:easycases}
    $$s_{(d,d)}\ast s_{(d,d)}=\rug{P}$$and $$s_{(d,d)}\ast s_{(d+1,d-1)}=\rug{\overline{P}}.$$
\end{corollary}

\begin{proof}From Corollary~\ref{max coeff}, we see that if $k=0$ or $k=1$, the maximum coefficient of a Schur function in $s_{(d,d)}\ast s_{(d+k,d-k)}$ is one. Also, $G_{-1}=G_0$ is all four part partitions, as the gaps in any partition are all at least zero, so $G_{-1} \cap P=P$ and $G_0\cap \overline{P}=\overline{P}$, which completes the proof. \end{proof}

\begin{corollary}\label{cor:support}
The support of $s_{(d,d)}\ast s_{(d+k+2,d-k-2)}$ is contained in the support of $s_{(d,d)}\ast s_{(d+k,d-k)}$ for $k >1$.
\end{corollary}

\begin{proof}The result follows if we show that all the index sets for rugs in $s_{(d,d)}\ast s_{(d+k+2,d-k-2)}$ are contained in index sets for rugs in $s_{(d,d)}\ast s_{(d+k,d-k)}$.

By definition we have that if $b \leq c$, then $G_{c} \subseteq G_b$, and  $D_{c} \subseteq D_b$, so for
$1 \leq a \leq k-1$,
$$\begin{array}{l}
    \Big. G_a \cap D_{a+k+2} \subseteq G_1 \cap D_{k+1}, \Big. \\
    \Big. D_{k+2} \cap P \subseteq D_{k} \cap P,  \Big.\\
    \Big. G_{k+2} \cap \overline{P} \subseteq G_{k} \cap \overline{P}. \Big.
\end{array}$$
Thus, all the index sets for $s_{(d,d)}\ast s_{(d+k+2,d-k-2)}$ that do not involve $X_k, Y_k,$ or $Z_k$ are contained in index sets for $s_{(d,d)}\ast s_{(d+k,d-k)}$.

For the remaining index sets, consider the following cases.
\begin{enumerate}
\item If $\lambda \in D_{k+2}\! \cap \! X_{k+2}$ and
$\lambda \notin D_{k}\! \cap \! X_k$, then all of $\lambda$'s gaps are odd, and either $k \leq \lambda_2 -\lambda_3$ or both $\lambda_1 -\lambda_2$ and $\lambda_3 -\lambda_4$ are greater than or equal to $k$.  Thus, all of $\lambda$'s gaps are at least one, and at least one gap is greater than or equal to $k$.

\item If $\lambda \in D_{k+2}\! \cap \! Y_{k+2}$ and $\lambda \notin D_{k}\! \cap \! Y_k$, then $k \leq \lambda_2 -\lambda_3$, and $\lambda_1 -\lambda_2$ and $\lambda_3 -\lambda_4$ are odd, so are both at least one.

\item If $\lambda \in D_{k+2}\! \cap \! Z_{k+2}$ and $\lambda \notin D_{k}\! \cap \! Z_k$, then $\lambda_1 -\lambda_2$ and $\lambda_3 -\lambda_4$ are both greater than or equal to $k$, and $\lambda_2 -\lambda_3$  is odd, so it is at least one.
\end{enumerate}
In each case, all the gaps are a minimum of one, and the double gaps are at least size $k\! +\! 1$, so $\lambda \in G_1 \cap D_{k+1}$ in each of the three cases above.  Thus,
$$\begin{array}{rcl}
    D_{k+2}\! \cap \! X_{k+2} &\subseteq &(D_{k}\! \cap \! X_k)\cup (G_1 \cap D_{k+1}), \\
    \bigg. D_{k+2}\! \cap \! Y_{k+2}&\subseteq &(D_{k}\! \cap \! Y_k)\cup (G_1 \cap D_{k+1}),  \bigg. \\
    D_{k+2}\! \cap \! Z_{k+2} &\subseteq &(D_{k}\! \cap \! Z_k)\cup (G_1 \cap D_{k+1}),
\end{array} $$
which completes the proof.  \end{proof}
}\comment{%remove this line to put back other version of sdd*sd+kd-k

\section{Proof of Proposition \ref{form prop}}\label{sec:proof}

In this section, we use a second theorem of \cite{Rosas} to prove Proposition \ref{form prop}, and the following definitions are required for the statement of the theorem.  For convenience, we denote $\mathbb{N} \cup  \{0\}$ as $\mathbb{N}$ for this discussion.
    \begin{definition} \label{defreach}
        Let $i,j,k$ and $l$ be non-negative integers.  We say $(k,l)$ is \emph{reachable} from $(i,j)$, written $(i,j) \rightsquigarrow (k,l)$, if there exists a path in $ \mathbb{N}^2$ from $(i,j)$ to $(k,l)$ that only uses steps that go north-west (N-W) or south-west (S-W), with the cardinal directions defined on $ \mathbb{N}^2$ in the natural way. 
    \end{definition}

\begin{remark}\label{defreachnote1} if $i \leq j$, the points reachable from $(i,j)$ are inside the triangle defined by the lines
    \begin{eqnarray*}
        L_1: && y = x-i+j \\
        L_2: && y = i+j-x \\
        L_3: && x = 0,
    \end{eqnarray*}
so the points are symmetric about the line $y=j$, as the triangle's apex sits on this line.  Further, the points in $L_1$ and $L_2$, for $ 0 \leq x \leq i$, are all reachable from $(i,j)$, as a path using only N-W steps hits every point on $L_2$, and similarly for $L_1$ using S-W steps.\end{remark}

\begin{remark}\label{defreachnote2} One cannot make a path using N-W or S-W steps to get from a point reachable from $(i,j)$ to a point directly above, below, or on either side, but any given horizontal line, one can take a N-W and a S-W step to get to alternating points west of a reachable point.\end{remark}

    \begin{definition}\label{rectangle}
        Let $x$, and $y$ be non-negative integers, and $R$ be a rectangle in $\mathbb{Z}^2$.
        \begin{eqnarray*}
            \Gamma_R(x,y) &=& |\{(i,j) \in R \! \cap \! \mathbb{N}^2 : (x,y) \! \rightsquigarrow \! (i,j) \} | \\
            &=& \textrm{number of points reachable from $(x,y)$ inside $R$}.         \end{eqnarray*}
    \end{definition}

\begin{example}
    Let $R$ be the rectangle with height and width $3$ and S-W vertex $(1,4)$.  Then the following picture demonstrates both the above definitions for $\Gamma_R(3,3)$.
   
    $$
      \scalebox{.5}{\includegraphics{reachable.pdf}}
    $$
    The darkened points are the points reachable from $(3,3)$, so we can see that $\Gamma_R(3,3)=2$.
\end{example}
}\comment{%remove this line to put back other version of sdd*sd+kd-k

We are now able to give the statement of the second theorem of \cite{Rosas} that we use.   
    \begin{theorem} \label{Rosas2}\cite[Theorem 1]{Rosas}
        Let $\mu, \nu$,\ and $\lambda$ be partitions of n, with $\mu=(\mu_1,\mu_2)$, $\nu=(\nu_1,\nu_2)$ and $\nu_2 \leq \mu_2$, and let $\lambda=(\lambda_1,\lambda_2,\lambda_3,\lambda_4)$ be a four part partition.  Let $M=min\{ \lambda_1 \! - \! \lambda_2, \lambda_3 \! - \! \lambda_4 \}$, and $m=max\{ \lambda_1 \! - \! \lambda_2, \lambda_3 \! - \! \lambda_4 \}$.    Then, for a given $\lambda$        \begin{eqnarray}
            \kroncoeff_{\mu\nu\lambda} = \big( \Gamma_P  \! -  \! \Gamma_N \big)(\nu_2, \mu_2 \! + \! 1),
        \end{eqnarray}
        where the rectangles $P$ and $N$ have S-W vertices $v_P=(\lambda_3 \! + \! \lambda_4, \lambda_2 \! + \! \lambda_4 \! + \! 1)$ and $v_N=(\lambda_3 \! + \! \lambda_4, \lambda_2 \! + \! \lambda_4 \! + \! m \! + \! 2)$ respectively, width $\lambda_2 \! - \! \lambda_3$ and height $M$.     \end{theorem}

The following lemma shows that the rectangle $N$ in the above formula may be replaced by its reflection across the line $y=\mu_2\!+\!1$.

    \begin{lemma} \label{reflec lemma}
        Let $\mu, \nu, \lambda, M, m$, and the rectangle $P$ be as in Theorem \ref{Rosas2}, and let $L$ be the line $y=\mu_2\! +\! 1$.  Define $N'$ to be the rectangle with width $\lambda_2 \! - \! \lambda_3$, height $M$, and S-W vertex
$$ v_{N'} = (v_{N'_{1}},v_{N'_{2}}) \! = \! (\lambda_3 \! + \! \lambda_4, \lambda_2 \! + \! \lambda_4 \! -\! (\mu_1 \! - \! \mu_2)).$$
Then
        $$\kroncoeff_{\mu\nu\lambda} = \big( \Gamma_P  \! -  \! \Gamma_{N'} \big)(\nu_2, \mu_2 \! + \! 1).$$
    \end{lemma}

\begin{proof}Let $N$ be the rectangle given in Theorem~\ref{Rosas2} Since $\mu$ and $\lambda$ are partitions of $n$, we have that
$$ \mu_2 \leq \left\lfloor \frac{n}{2} \right\rfloor \leq max\{ \lambda_1 \! + \! \lambda_4, \lambda_2 \! + \! \lambda_3 \}. $$
We also have that
$$ v_{N_2}=\lambda_2 \! + \! \lambda_4 \! + \! m \! + \! 2 = max\{ \lambda_1 \! + \! \lambda_4, \lambda_2 \! + \! \lambda_3 \} + 2
    > \mu_2 \! +\! 1,$$
so the rectangle $N$ lies strictly above $L$.  The S-W vertex of $N$'s reflection is the reflection of $N$'s N-W vertex, $(v_{N_1}, v_{N_2}\!+\!M)$; the reflection of the N-W vertex is given by $(v_{N_1}, \mu_2\!+\!1 -(v_{N_2}\!+\!M-\mu_2\!-\!1))$, as the reflection is across the horizontal line L.  We also have that
$$\begin{array}{rcll}
    2\mu_2\!+\!2 \! - (\!v_{N_2} \! + \! M)& = & 2\mu_2 \! -(\lambda_2 \! + \! \lambda_4 \! + \! m \! + \! M) \\
        &=& 2\mu_2 \! -(\lambda_2 \! + \! \lambda_4 \! + \! \lambda_1 \! - \! \lambda_2 \! + \! \lambda_3 \! - \! \lambda_4) \\
        &=& n\! - (\mu_1 \! - \! \mu_2) -\! \lambda_1 \! - \! \lambda_3 & (\textrm{ as } \mu \vdash n)\\
        &=& \lambda_2\! +\! \lambda_4- (\mu_1 \! - \! \mu_2), & (\textrm{ as } \lambda \vdash n)
\end{array}$$
so the reflection of $N$ across $L$ is $N'$, as the reflection has the same dimensions and S-W vertex $v_{N'}$.  Since $\nu_2 \leq \mu_2\! +\!1$, the points reachable from $(\nu_2,\mu_2+1)$ are symmetric about L, so any reachable point in $N$ has a reachable reflection in $N'$, and so
$$ \Gamma_N(\nu_2, \mu_2 \! + \! 1)=\Gamma_{N'}(\nu_2, \mu_2 \! + \! 1). $$
The result follows. \end{proof}
}\comment{%remove this line to put back other version of sdd*sd+kd-k
One way to calculate $\Gamma_R(\nu_2,\mu_2 \! + \! 1)$ is to decompose the rectangle $R$ into a number of horizontal line segments, and then $$\Gamma_R(\nu_2,\mu_2 \! + \! 1)=\sum_i \Gamma_{R_i}(\nu_2,\mu_2 \! + \! 1),$$ where $R_i$ is the $i$th line segment of $R$.  The next lemma provides a piecewise formula for calculating $\Gamma_{R_i}(\nu_2,\mu_2 \! + \! 1)$ when $R_i$ is a line segment.

\begin{lemma} \label{line seg lemma}
    Let $a,b$, and $c$ be non-negative integers, and $R_i$ be a line segment in $\mathbb{N}^2$ with western vertex $(a,c)$ and width $b$.  Define the lines
    \begin{eqnarray*}
        L_1: && y = x-\nu_2+\mu_2+1 \\
        L_2: && y = \nu_2+\mu_2+1-x \\
        L_{R_i}: && y = c,
    \end{eqnarray*}
 and let $x_{R_i}$ to be the $x$-coordinate of the intersection of $L_{R_i}$ and $L_1$ if $c \leq \mu_2\! +\! 1$, and the intersection of $L_{R_i}$ and $L_2$ otherwise.  
    \begin{enumerate}
        \item If $x_{R_i} < a$, then $\Gamma_{R_i}(\nu_2, \mu_2 \! + \! 1)=0$.

        \item If $a  \leq  x_{R_i}  \leq  a\! + \! b$, then $ \Gamma_{R_i}(\nu_2, \mu_2 \! + \! 1)= \left\lceil \frac{x_{R_i}\! +\! 1\! -\! a}{2} \right\rceil.$

        \item If $a \! + \! b  <  x_{R_i}$, then
        $ \Gamma_{R_i}(\nu_2, \mu_2 \! + \! 1)= \left\lceil \frac{b}{2} \right\rceil +\charf{a \! + \! x_{R_i} \textrm{ even}}
            \charf{ b \textrm{ even}}.$
    \end{enumerate}
\end{lemma}

\begin{proof}From Remark~\ref{defreachnote1}, we have that the region of points reachable from $(\nu_2,\mu_2+1)$ is a triangle with $L_1$ and $L_2$ forming the eastern boundary, so $(x_{R_i},c)$ is the point of intersection of $L_{R_i}$, and that boundary.  Again by Remark~\ref{defreachnote1}, $(x_{R_i},c)$ is reachable from $(\nu_2,\mu_2+1)$, and there no are points reachable from $(\nu_2,\mu_2+1)$ east of $(x_{R_i},c)$ on the line $L_{R_i}$.  The line segment ${R_i}$ and $(x_{R_i},c)$ both sit on the line $L_{R_i}$ by definition, so ${R_i}$ may be strictly east of $(x_{R_i},c)$, contain $(x_{R_i},c)$ or be strictly west of $(x_{R_i},c)$.  We determine $\Gamma_{R_i}(\nu_2,\mu_2+1)$ in each case.
\end{proof}}\comment{\begin{proof}(continued)%remove this line to put back other version of sdd*sd+kd-k

\begin{enumerate}
    \item If ${R_i}$ is strictly east of $(x_{R_i},c)$, then $x_{R_i} < a$, and
        $\Gamma_{R_i}(\nu_2, \mu_2 \! + \! 1)=0$
        as all the points reachable from $(\nu_2, \mu_2 \! + \! 1)$ are west of $(x_{R_i},c)$ on $L_{R_i}$.

    \item  If $(x_{R_i},c)$ is contained in ${R_i}$, then $a \leq x_{R_i} \leq a+b$, and there are $x_{R_i}\! -\! (a\! -\! 1)$ points of ${R_i}$ inside the region of points reachable from $(\nu_2,\mu_2\! +\! 1)$.  Since the reachable points alternate along $L_{R_i}$ starting at $(x_{R_i},c) \in {R_i}$, we have
        $$\Gamma_{R_i}(\nu_2, \mu_2 \! + \! 1)=\left\lceil \frac{x_{R_i}\! -\! (a\! -\! 1)}{2} \right\rceil.$$

    \item  If ${R_i}$ is strictly west of $(x_{R_i},c)$, then $a+b < x_{R_i}$, and ${R_i}$ is fully inside the region of points reachable from $(\nu_2,\mu_2\!+\!1)$.  From Remark~\ref{defreachnote2}, the points reachable from $(\nu_2,\mu_2\!+\!1)$ alternate along $L_{R_i}$ with $(x_{R_i},c)$ being the most easterly.

Now, either the western point of ${R_i}$, $(a, c)$, is reachable from $(\nu_2,\mu_2+1)$ or $(a+1, c) \in {R_i}$ is reachable.  If $(a, c)$ is reachable from $(\nu_2,\mu_2+1)$, then $x_{R_i}\equiv a\mod 2$, so by Lemma \ref{ceil lemma} part 2
    $$\Gamma_{R_i}(\nu_2,\mu_2+1)=\left\lceil \frac{b\! + \! 1}{2} \right\rceil=\left\lceil \frac{b}{2} \right\rceil +\charf{b \textrm{ even}},$$
 otherwise $(a\!+\!1, c)$ is reachable, and $x_{R_i}\not\equiv a\mod 2$, so $$\Gamma_{R_i}(\nu_2,\mu_2+1)=\left\lceil \frac{b}{2} \right\rceil.$$
Thus, by Lemma \ref{propfunc} part 2,
    \begin{eqnarray*}
        \Gamma_{R_i}(\nu_2,\mu_2+1) &=& \bigg. \charf{a \equiv x_{R_i} \mod 2}
            \left( \left\lceil \frac{b}{2} \right\rceil +\charf{b \textrm{ even}} \right) \bigg.\\
           && + \charf{a \not\equiv x_{R_i} \mod 2}\left\lceil \frac{b}{2} \right\rceil \\ \\
        &=& \left\lceil \frac{b}{2} \right\rceil + \charf{a \equiv x_{R_i} \mod 2}
            \charf{ b \textrm{ even}  }  \\
        &=& \left\lceil \frac{b}{2} \right\rceil + \charf{a \! + \! x_{R_i} \textrm{ even}}
            \charf{ b \textrm{ even}  }.  
    \end{eqnarray*}
\end{enumerate}\end{proof}

For the remainder of this section, we focus on Kronecker products of the form $s_{(d,d)}\ast s_{(d+\! k,d-\! k)}$, and since $d\!-\!k \leq d$, the formula from Lemma \ref{reflec lemma} becomes
$$\kroncoeff_{\mu\nu\lambda} = \big( \Gamma_P  \! -  \! \Gamma_{N'} \big)(d\! -\! k, d \! + \! 1).$$
}\comment{

The formula only requires us to consider points reachable from $(d \! - \! k,d \! + \! 1)$, so we will refer to such points simply as \emph{reachable points}, and write $\Gamma_R$ in place of $\Gamma_R(d \! - \! k,d \! + \! 1)$ for any rectangle $R$.

\begin{proposition} \label{coeff id prop}
     Let $\mu, \nu$, and $\lambda$ be as in Theorem \ref{Rosas2}, and $M=min\{ \lambda_1 \! - \! \lambda_2, \lambda_3 \! - \! \lambda_4 \}$.      \begin{enumerate}
        \item If $k \leq  \lambda_2 \! - \! \lambda_3 \! + \! M$, $M \leq k$, and $\lambda_2 \! - \! \lambda_3 \leq k$, then
        $\kroncoeff_{\mu\nu\lambda}= \left\lceil \frac{\lambda_2 \! - \! \lambda_3 \! + \! M \! - \! k \!}{2} \right\rceil
            + \charf{ \lambda_2 \! - \! \lambda_3 \! + \! M \! - \! k \textrm{ even}  }.$

        \item If $M \leq k$, and $k < \lambda_2 \! - \! \lambda_3$, then
            $\kroncoeff_{\mu\nu\lambda} = \left\lceil \frac{ M }{2} \right\rceil
            + \charf{ \lambda_2 \! - \! \lambda_3 \!  - \! k \textrm{ even}  }\charf{ M  \textrm{ even} }.$

        \item If $k < M$ and $\lambda_2 \! - \! \lambda_3 \leq k$, then
            $ \kroncoeff_{\mu\nu\lambda} = \left\lceil \frac{\lambda_2 \! - \! \lambda_3}{2} \right\rceil +
                        \charf{ M \! - \! k \textrm{ even}}
                        \charf{ \lambda_2 \! - \! \lambda_3 \textrm{ even}}.$

        \item If $k< M$ and $k< \lambda_2 \! - \! \lambda_3$, then
        $ \kroncoeff_{\mu\nu\lambda}=\left\lceil \frac{k}{2} \right\rceil
            \! + \! \charf{ \lambda_2 \! - \! \lambda_3  \textrm{ even}  } \charf{ M  \textrm{ even}}
            \Big[ \charf{ \! k \textrm{ even}  } - \charf{ \! k \textrm{ odd}  } \Big]. $

        \item If $ \lambda_2 \! - \! \lambda_3 \! + \! M < k $, then $\kroncoeff_{\mu\nu\lambda}=0$.
    \end{enumerate}
\end{proposition}

\begin{proof}Let $P$ be the rectangle as in Theorem \ref{Rosas2}, and let $N'$ be the rectangle as in Lemma \ref{reflec lemma}.  Then $P$ and $N'$ are rectangles with width $\lambda_2 \! - \! \lambda_3$, height $M$ and S-W vertices
\begin{eqnarray*}
    v_P &=& (\lambda_3 \! + \! \lambda_4, \lambda_2 \! + \! \lambda_4 \! + \! 1) \\
    v_{N'} & = &(\lambda_3 \! + \! \lambda_4, \lambda_2 \! + \! \lambda_4 \! -\! (d \! - \! d))
            =(\lambda_3 \! + \! \lambda_4, \lambda_2 \! + \! \lambda_4 ),
\end{eqnarray*}
respectively.  Since $v_{P_2} = v_{N'_2}\! +\! 1$, and $P$ and $N'$ have the same dimensions, they overlap each other except for the top edge of $P$, and the bottom edge of $N'$.  Define the top edge of $P$ to be $T$, and the bottom edge of $N'$ to be $B$.  $T$ and $B$ are line segments with width $b=\lambda_2 \! - \! \lambda_3$, and S-W vertices
    \begin{eqnarray*}
            (a,c_T) &=& (\lambda_3 \! + \! \lambda_4, \lambda_2 \! + \! \lambda_4 \! + \! M \! + \! 1) \\
            (a,c_B) &=& (\lambda_3 \! + \! \lambda_4, \lambda_2 \! + \! \lambda_4 \! ),
    \end{eqnarray*}
respectively.  The area where $P$ and $N'$ overlap can be described by either of $P \setminus T$ or $N' \setminus B$, so $\Gamma_{P \setminus T} = \Gamma_{N' \setminus B}.$
Thus,
$$\begin{array}{rclr}
    \kroncoeff_{\mu\nu\lambda} & = & \Gamma_P  \! -  \! \Gamma_{N'}  \\
    & = & \Gamma_{P \setminus T}  \! +  \! \Gamma_T - \Gamma_{N' \setminus B} \! - \! \Gamma_B   \\
    & = & \Gamma_T \! - \! \Gamma_B.
\end{array}$$

We apply Lemma \ref{line seg lemma} to calculate $\Gamma_T$ and $\Gamma_B$ for the line segments $T$ and $B$. Define the points $(x_T, c_T)$ and $(x_B,c_B)$, and the  lines as in Lemma \ref{line seg lemma}, with $\nu_2=d-k$, and $\mu_2=d$:
\end{proof}}\comment{\begin{proof}(continued)%remove this line to put back other version of sdd*sd+kd-k
    \begin{eqnarray*}
        L_1: && y = x+k+1 \\
        L_T: && y = c_T = \lambda_2 \! + \! \lambda_4 \! + \! M \! + \! 1 \\
        L_B: && y = c_B = \lambda_2 \! + \! \lambda_4.
    \end{eqnarray*}
Since
$$c_T= \lambda_2 \! + \! \lambda_4 \! + \! M \! + \! 1 = min\{ \lambda_1 \! + \! \lambda_4, \lambda_2 \! + \! \lambda_3 \}  \! + 1 \leq d \! + \! 1, \quad (\textrm{as } \lambda\vdash 2d) $$

we take $x_T$ to be the $x$-coordinate of the intersection of $L_1$ and $L_T$, and similarly, $x_B$ is the $x$-coordinate of the intersection of $L_1$ and $L_B$, as $c_B < c_T \leq d+1$.  We then have that

    \begin{eqnarray*}
        x_T  \! + \! k \! + \! 1 = \lambda_2 \! + \! \lambda_4 \! + \! M \! +1 & \Leftrightarrow & x_T = \lambda_2 \!
            + \! \lambda_4 \! + \! M \! - \! k, \quad \\ \\
        x_B  \! + \! k \! + \! 1 = \lambda_2 \! + \! \lambda_4 & \Leftrightarrow & x_B = \lambda_2 \! + \! \lambda_4 \! - \! k \! -1.
    \end{eqnarray*}

From Lemma \ref{line seg lemma}, we note that the third case for $\Gamma_B$ requires $a+b < x_B$, but
    $$ \lambda_2 \! + \! \lambda_4 \! =\! a\! +\! b < x_B \! =\! \lambda_2 \! + \! \lambda_4 \! - \! k \! -1$$
is a contradiction, so this case cannot occur for $\Gamma_B$.  We combine the three cases for $\Gamma_T$ with the remaining two cases for $\Gamma_B$ to calculate $\kroncoeff_{\mu\nu\lambda}= \Gamma_T \! - \! \Gamma_B$.
\begin{enumerate}

\item Take $a \leq x_T \leq a+b$ and $x_B < a$.  We have the following equivalences for these conditions.
    \begin{eqnarray*}
        \lambda_3 \! + \! \lambda_4 = a \leq x_T = \lambda_2 \! + \! \lambda_4 \! + \! M \! - \! k & \Leftrightarrow &
        \Big. k \leq  \lambda_2 \! - \! \lambda_3 \! + \! M, \Big. \\
        \lambda_2 \! + \! \lambda_4 \! + \! M \! - \! k = x_T \leq a+b = \lambda_2 \! + \! \lambda_4 & \Leftrightarrow &
        \Big. M  \leq  k \Big. , \qquad \qquad\\
        \lambda_2 \! + \! \lambda_4 \! - \! k \! -1 = x_B < a = \lambda_3 \! + \! \lambda_4 & \Leftrightarrow &
        \Big. \lambda_2 \! - \! \lambda_3 \leq k, \Big.
    \end{eqnarray*}
and these are the required conditions for the first part.

From Lemma \ref{line seg lemma}, we have that
    $$\kroncoeff_{\mu\nu\lambda}=\left\lceil \frac{x_T\! +\! 1\! -\! a}{2} \right\rceil -0
        = \left\lceil \frac{\lambda_2 \! - \! \lambda_3 \! + \! M \! - \! k \! +\! 1\!}{2} \right\rceil,$$
from which the result follows using Lemma \ref{ceil lemma} part 2.

\item Take $a \leq x_T \leq a+b$ and $a \leq x_B $.  Since $x_B \leq x_T$, the condition $a \leq x_T$ is implicit in $a \leq x_B $.  We use the case above for the condition equivalent to $x_T \leq a+b$.  For $a \leq x_B $, we have $ k <  \lambda_2 \! - \! \lambda_3.$
These are the required conditions for the second part.  Using Lemma \ref{line seg lemma} and Lemma \ref{ceil lemma}, we have
\end{enumerate}\end{proof}}\comment{\begin{proof}\begin{enumerate}(continued)%remove this line to put back other version of sdd*sd+kd-k
\begin{eqnarray*}
    \kroncoeff_{\mu\nu\lambda} &=& \left\lceil \frac{x_T\! +\! 1\! -\! a}{2}\right\rceil - \left\lceil \frac{x_B\! +\! 1\! -\! a}{2} \right\rceil \\
        &=& \left\lceil \frac{\lambda_2 \! - \! \lambda_3 \! + \! M \! - \! k \! +\! 1\!}{2} \right\rceil -
            \left\lceil \frac{\lambda_2 \! - \! \lambda_3 \! - \! k \!}{2} \right\rceil \\
        &=& \frac{M + \charf{ \lambda_2 \! - \! \lambda_3 \! + \! M \! - \! k \! +\! 1 \textrm{ odd}  }
            + 1 - \charf{ \lambda_2 \! - \! \lambda_3 \! - \! k  \textrm{ odd}  }}{2}\\
        &=& \left\lceil \frac{ M }{2} \right\rceil
            + \charf{ \lambda_2 \! - \! \lambda_3 \!  - \! k \textrm{ even}  }\charf{ M  \textrm{ even} }.
\end{eqnarray*}

\item Take $a+b < x_T$ and $x_B < a$.  Since
    \begin{eqnarray*}
        \lambda_2 \! + \! \lambda_4 =a\! +\! b < x_T=\lambda_2 \! + \! \lambda_4\! + \! M \! - \! k
            & \Leftrightarrow & k< M,
    \end{eqnarray*}
the conditions can be written as $k< M$ and $\lambda_2 \! - \! \lambda_3 \leq k$, which are as required for the third part.  Lemma \ref{line seg lemma} gives
\begin{eqnarray*}
    \kroncoeff_{\mu\nu\lambda} &=& \left\lceil \frac{b}{2} \right\rceil +\charf{a \! + \! x_R \textrm{ even}}
            \charf{ b \textrm{ even}}\\
       &=& \left\lceil \frac{\lambda_2 \! - \! \lambda_3}{2} \right\rceil +
            \charf{\lambda_2 \! + \lambda_3 \! +\! 2\lambda_4\! + \! M \! - \! k \textrm{ even}}
            \charf{ \lambda_2 \! - \! \lambda_3 \textrm{ even}}\\
        &=&\left\lceil \frac{\lambda_2 \! - \! \lambda_3}{2} \right\rceil +
            \charf{ M \! - \! k \textrm{ even}}
            \charf{ \lambda_2 \! - \! \lambda_3 \textrm{ even}},
\end{eqnarray*}
as $\charf{ \lambda_2 \! - \! \lambda_3 \textrm{ even}} =\charf{ \lambda_2 \! + \! \lambda_3 \textrm{ even}}$.

\item Take $a+b < x_T$ and $a \leq x_B$.  These can be written as the conditions for the fourth part: $k< M$ and $k< \lambda_2 \! - \! \lambda_3$. By Lemma \ref{line seg lemma} and Lemma \ref{ceil lemma}, we have
\begin{eqnarray*}
    \kroncoeff_{\mu\nu\lambda} &=&  \left\lceil \frac{\lambda_2 \! - \! \lambda_3}{2} \right\rceil -
            \left\lceil \frac{\lambda_2 \! - \! \lambda_3 \! - \! k \!}{2} \right\rceil
            +\charf{ M \! - \! k \textrm{ even}}
            \charf{ \lambda_2 \! - \! \lambda_3 \textrm{ even}}\\
        &=& \frac{ k \! + \! \charf{ \lambda_2 \! - \! \lambda_3  \textrm{ odd}  }\! - \! \charf{ \lambda_2 \! - \! \lambda_3 \!  - \! k \textrm{ odd}  }}{2}
            +\charf{ M \! - \! k \textrm{ even}}
            \charf{ \lambda_2 \! - \! \lambda_3 \textrm{ even}}\\
        &=& \left\lceil \frac{k}{2} \right\rceil
            \! + \! \charf{ \lambda_2 \! - \! \lambda_3  \textrm{ even}  } \charf{ M  \textrm{ even}}
            \Big[ \charf{ \! k \textrm{ even}  } - \charf{ \! k \textrm{ odd}  } \Big].
\end{eqnarray*}

\item Take $x_T < a$.  This requirement can be rewritten as  $\lambda_2 \! - \! \lambda_3 + \! M< k.$
Lemma \ref{line seg lemma} gives that $\Gamma_T=0$.  We also have that $x_B < x_T<a$, so $\Gamma_B=0$, and so
$\kroncoeff_{\mu\nu\lambda}=0.$
\end{enumerate}\end{proof}

}\comment{%remove this line to put back other version of sdd*sd+kd-k

We are now is a position to prove Proposition \ref{form prop}, which we restate here for convenience.  Recall that for a four part  partition $\lambda$ we have  $g_\lambda  =min\{\lambda_1-\lambda_2, \lambda_2-\lambda_3, \lambda_3-\lambda_4 \}$,  $d_\lambda = min \{ \lambda_1-\lambda_3, \lambda_2-\lambda_4 \}$ and $M= min\{\lambda_1  \! -  \! \lambda_2, \lambda_3  \! -  \! \lambda_4\}$.\\

\begin{proposition*}
Let $ \mu = (d,d)$, $\nu = (d \! + \! k,d \! - \! k)$, and $\lambda = (\lambda_1, \lambda_2, \lambda_3, \lambda_4)$ all be partitions of $2d$.    The Kronecker product coefficient, $\kroncoeff_{\mu\nu\lambda}$, is given by 
        \begin{eqnarray*}
                \kroncoeff_{\mu\nu\lambda} & = &
                \sum_{ \stackrel{1 \leq \, a \, \leq k-1  }{ a \; odd}}
                \charf{ a  \! \leq g_\lambda  }
                \charf{ a  \! +  \! k \leq d_\lambda }
                \\  & & +  \charf{ k \textrm{ even}   }
                \charf{ k  \! \leq  \! d_\lambda }
                \left( \begin{array}{l}
                    \Big. \charf{ M  \! <  \! k }
                    \charf{ \lambda_2  \! -  \! \lambda_3  \! <  \! k }
                    \charf{ M \textrm{ odd}  }
                    \charf{ \lambda_2  \! -  \! \lambda_3 \textrm{ odd}   } \Big. \\
                  + \Big. \charf{ M \textrm{ even}   }
                    \charf{ \lambda_2  \! -  \! \lambda_3 \textrm{ even}   } \Big.
                \end {array} \right)
                \\
                \\ & & +  \charf{ k \textrm{ odd}   }
                \charf{ k  \! \leq  \! d_\lambda }
                \left( \begin{array}{l}
                    \Big. \charf{ \lambda_2  \! -  \! \lambda_3  \! <  \! k }
                    \charf{ M \textrm{ odd}  }
                    \charf{ \lambda_2  \! -  \! \lambda_3 \textrm{ even}   } \Big.
\\
                     + \Big. \charf{ M  \! <  \! k }
                    \charf{ M \textrm{ even}  }
                    \charf{ \lambda_2  \! -  \! \lambda_3 \textrm{ odd}  } \Big.
\\
                    + \Big. \charf{ k  \! \leq  \! g_\lambda }
                    \charf{ M \textrm{ odd or } \lambda_2  \! -  \! \lambda_3 \textrm{ odd}  } \Big.
                \end{array} \right).
        \end{eqnarray*}
        \end{proposition*}

\begin{proof}Using Proposition \ref{coeff id prop}, we write $\kroncoeff_{\mu\nu\lambda}$ using characteristic functions to track the cases:
    \begin{eqnarray*}
        \kroncoeff_{\mu\nu\lambda} &=&  \charf{ k \! \leq \! \lambda_2  \! -  \! \lambda_3  \! +  \! M }
            \charf{ M  \! \leq \! k } \charf{ \lambda_2 \! - \! \lambda_3 \! \leq \! k }
            \Bigg( \left\lceil \frac{\lambda_2 \! - \! \lambda_3 \! + \! M \! - \! k \!}{2} \right\rceil
            + \charf{ \lambda_2 \! - \! \lambda_3 \! + \! M \! - \! k \textrm{ even}  } \Bigg)\\
        && + \charf{ M \! \leq \! k } \charf{ k \! < \! \lambda_2 \! - \! \lambda_3 }
            \Bigg( \left\lceil \frac{ M }{2} \right\rceil
            + \charf{ \lambda_2 \! - \! \lambda_3 \!  - \! k \textrm{ even}  }\charf{ M  \textrm{ even} } \Bigg)\\
        && + \charf{ k \! < \! M } \charf{ \lambda_2 \! - \! \lambda_3 \! \leq \! k }
            \Bigg( \left\lceil \frac{\lambda_2 \! - \! \lambda_3}{2} \right\rceil +
            \charf{ M \! - \! k \textrm{ even}}\charf{ \lambda_2 \! - \! \lambda_3 \textrm{ even}}\Bigg)\\
        && + \charf{ k \! < \! M }\charf{ k \! < \! \lambda_2 \! - \! \lambda_3 }
            \Bigg( \left\lceil \frac{k}{2} \right\rceil
            \! + \! \charf{ \lambda_2 \! - \! \lambda_3  \textrm{ even}  } \charf{ M  \textrm{ even}}
            \Big[ \charf{ \! k \textrm{ even}  } - \charf{ \! k \textrm{ odd}  } \Big] \Bigg).
    \end{eqnarray*}
We simplify $\kroncoeff_{\mu\nu\lambda}$ in two parts: a sum of the ceiling functions and a sum of the parity characteristic functions.
\end{proof}}\comment{\begin{proof}(continued)%remove this line to put back other version of sdd*sd+kd-k

To simplify the sum of the ceiling functions, which we denote $\Sigma$, we use Lemma \ref{ceil lemma} part 3 to change the ceiling functions in to sums, and we note three facts.  First, the index $a$ in each sum cannot be greater than $k$, so there are at most $\left\lceil \frac{k}{2} \right\rceil$ non-zero summands in each sum; we can therefore bound the index of summation by $k$.  In the last sum, we have that all the terms are equal to $1$, as $\charf{ a \leq k }=1$ for $1\leq a\leq k$.  Second,
$$\begin{array}{l}
    \Big. \charf{ M  \! < \! a } \charf{ \lambda_2 \! - \! \lambda_3 \! \leq \! k }
            \charf{ a \leq \lambda_2 \! - \! \lambda_3 \! + \! M \! - \! k }=0,  \Big.\\
    \Big. \charf{ M  \! \leq \! k } \charf{ \lambda_2 \! - \! \lambda_3 \! < \! a }
            \charf{ a \leq \lambda_2 \! - \! \lambda_3 \! + \! M \! - \! k }=0 \Big. ,
\end{array}$$
so in the first sum, $\charf{ r \! \leq \! k } =\charf{ a \! \leq \! r \! \leq \! k },$
where $r=M$ or $r=\lambda_2 \! - \! \lambda_3$.  Third, for non-negative integers $t$ and $s$,
$ k < t \textrm{ or } k < s \Rightarrow a\! + \! k \leq t\! +\! s,$ as the index obeys $a\leq k$, so by Lemma \ref{propfunc} part 5, we may multiply all the summands of $\Sigma$ by $\charf{ a \! + \! k \!\leq \! \lambda_2 \! - \! \lambda_3 \! + \! M  }$.

Applying this all to $\Sigma$, we get
    \begin{eqnarray*}
        \Sigma &=&  \charf{ k \! \leq \! \lambda_2  \! -  \! \lambda_3  \! +  \! M }
            \charf{ M  \! \leq \! k } \charf{ \lambda_2 \! - \! \lambda_3 \! \leq \! k }
            \left\lceil \frac{\lambda_2 \! - \! \lambda_3 \! + \! M \! - \! k \!}{2} \right\rceil
\\
        && + \charf{ M \! \leq \! k } \charf{ k \! < \! \lambda_2 \! - \! \lambda_3 }
             \left\lceil \frac{ M }{2} \right\rceil
\\
        && + \charf{ k \! < \! M } \charf{ \lambda_2 \! - \! \lambda_3 \! \leq \! k }
            \left\lceil \frac{\lambda_2 \! - \! \lambda_3}{2} \right\rceil
\\
        && + \charf{ k \! < \! M }\charf{ k \! < \! \lambda_2 \! - \! \lambda_3 }
            \left\lceil \frac{k}{2} \right\rceil
\end{eqnarray*}
\begin{eqnarray*}
        &  = &  \sum_{ \stackrel{1 \leq \, a \, \leq k }{ a \; odd}}
            \charf{ a \! \leq \! M  \! \leq \! k } 
            \charf{ a \! \leq \! \lambda_2 \! - \! \lambda_3 \! \leq \! k }
            \charf{ a \! + \! k \!\leq \! \lambda_2 \! - \! \lambda_3 \! + \! M  }
\\
        && + \sum_{ \stackrel{1 \leq \, a \, \leq k }{ a \; odd}}\charf{ a \! \leq \! M \! \leq \! k }
            \charf{ k \! < \! \lambda_2 \! - \! \lambda_3 }
            \charf{ a \! + \! k \!\leq \! \lambda_2 \! - \! \lambda_3 \! + \! M  }
\\
        && + \sum_{ \stackrel{1 \leq \, a \, \leq k }{ a \; odd}}
             \charf{ k \! < \! M } \charf{ a \! \leq \! \lambda_2 \! - \! \lambda_3 \! \leq \! k }
            \charf{ a \! + \! k \!\leq \! \lambda_2 \! - \! \lambda_3 \! + \! M  }
\\
        && + \sum_{ \stackrel{1 \leq \, a \, \leq k }{ a \; odd}}
            \charf{ k \! < \! M }\charf{ k \! < \! \lambda_2 \! - \! \lambda_3 }
            \charf{ a \! + \! k \!\leq \! \lambda_2 \! - \! \lambda_3 \! + \! M  }
\\
    &=& \sum_{ \stackrel{1 \leq \, a \, \leq k }{ a \; odd}} \charf{ a \! \leq \! M  }
        \charf{ a \! \leq \! \lambda_2 \! - \! \lambda_3  }
        \charf{ a \! + \! k \!\leq \! \lambda_2 \! - \! \lambda_3 \! + \! M  }
\\
    &=& \sum_{ \stackrel{1 \leq \, a \, \leq k-1 }{ a \; odd}} \charf{ a \! \leq \! g_\lambda  }
        \charf{ a \! + \! k \!\leq \! d_\lambda  } + \charf{ k  \textrm{ odd}}\charf{ k \! \leq \! g_\lambda  },
    \end{eqnarray*}
as for $0 \leq b$, 
$\charf{ b \! \leq \! M  } \charf{ b \! \leq \! \lambda_2 \! - \! \lambda_3  }= \charf{ b \! \leq \! g_\lambda  } \; \textrm{ and } \; \charf{ b \! \leq \! \lambda_2  \! -  \! \lambda_3  \! +  \! M }=\charf{ b \! \leq \! d_\lambda }.$
The term produced in the last transform will be used in the simplification of the case when $k$ is odd  below, and we will refer to that term as $\alpha$.
\end{proof}}\comment{\begin{proof}(continued)%remove this line to put back other version of sdd*sd+kd-k

Meanwhile, the sum of characteristic functions can be dealt with in two cases: when $k$ is even, and when $k$ is odd, and we denote these sums as $K_{\textrm{even}}$ and $K_{\textrm{odd}}$ respectively.  We begin by taking $k$ even.  The calculations that follow are similar to the manipulations of characteristic function in Proposition \ref{coeff id prop}, and so the details are omitted.
    \begin{eqnarray*}
        K_{\textrm{even}} &=&  \Big.\charf{ k \! \leq \! \lambda_2  \! -  \! \lambda_3  \! +  \! M }
            \charf{ M  \! \leq \! k } \charf{ \lambda_2 \! - \! \lambda_3 \! \leq \! k }
            \charf{ \lambda_2 \! - \! \lambda_3 \! + \! M \! - \! k \textrm{ even}  }\Big.\\
        && + \Big.\charf{ M \! \leq \! k } \charf{ k \! < \! \lambda_2 \! - \! \lambda_3 }
            \charf{ \lambda_2 \! - \! \lambda_3 \!  - \! k \textrm{ even}  }\charf{ M  \textrm{ even} } \Big.\\
        && + \Big.\charf{ k \! < \! M } \charf{ \lambda_2 \! - \! \lambda_3 \! \leq \! k }
            \charf{ M \! - \! k \textrm{ even}}\charf{ \lambda_2 \! - \! \lambda_3 \textrm{ even}}\Big.\\
        && + \Big.\charf{ k \! < \! M }\charf{ k \! < \! \lambda_2 \! - \! \lambda_3 }
            \! \charf{ \lambda_2 \! - \! \lambda_3  \textrm{ even}  } \charf{ M  \textrm{ even}}\Big.
\\\\
        &=&\Big.\charf{ k \! \leq \! d_\lambda }
            \charf{ M  \! \leq \! k } \charf{ \lambda_2 \! - \! \lambda_3 \! \leq \! k }
            \charf{ M \textrm{ odd}  }\charf{ \lambda_2 \! - \! \lambda_3 \textrm{ odd}  }\Big.\\
      && + \Big.\charf{ k \! \leq \! d_\lambda }
            \charf{ M \textrm{ even}}\charf{ \lambda_2 \! - \! \lambda_3 \textrm{ even}}\Big. .
    \end{eqnarray*}
The inequalities $M  \! \leq \! k$ and $\lambda_2 \! - \! \lambda_3 \! \leq \! k$ above may be made strict since $k$ is even in this case, and $M$ and $\lambda_2  \! -  \! \lambda_3$ must both be odd for that term.

Take $k$ odd. These calculations are also similar to those in Proposition \ref{coeff id prop}, and so the details are omitted. 
    \begin{eqnarray*}
        K_{\textrm{odd}} + \alpha &=&  \charf{ k \! \leq \! \lambda_2  \! -  \! \lambda_3  \! +  \! M }
            \charf{ M  \! \leq \! k } \charf{ \lambda_2 \! - \! \lambda_3 \! \leq \! k }
            \Big. \charf{ \lambda_2 \! - \! \lambda_3 \! + \! M  \textrm{ odd}  } \Big.\\
        && + \charf{ M \! \leq \! k } \charf{ k \! < \! \lambda_2 \! - \! \lambda_3 }
            \Big. \charf{ \lambda_2 \! - \! \lambda_3  \textrm{ odd}  }\charf{ M  \textrm{ even} } \Big.\\
        && + \charf{ k \! < \! M } \charf{ \lambda_2 \! - \! \lambda_3 \! \leq \! k }
            \Big. \charf{ M  \textrm{ odd}}\charf{ \lambda_2 \! - \! \lambda_3 \textrm{ even}}\Big.\\
        && - \charf{ k \! < \! M }\charf{ k \! < \! \lambda_2 \! - \! \lambda_3 }
            \Big. \charf{ \lambda_2 \! - \! \lambda_3  \textrm{ even}  } \charf{ M  \textrm{ even}} \Big.\\
        && + \charf{ k \! \leq \! M  } \charf{ k \! \leq \! \lambda_2 \! - \! \lambda_3  }\\\\
        &=& \charf{ k \! \leq \! d_\lambda }
            \left(\begin{array}{l}
                \charf{ M  \! \leq \! k }
                \Big. \charf{ M  \textrm{ even}  }\charf{ \lambda_2 \! - \! \lambda_3 \textrm{ odd}  } \Big. \\
                + \charf{ \lambda_2 \! - \! \lambda_3 \! \leq \! k }
                \Big. \charf{ M  \textrm{ odd}  } \Big.\charf{ \lambda_2 \! - \! \lambda_3 \textrm{ even}  }
            \end{array}\right)\\
        && + \Big. \charf{ k \! \leq \! g_\lambda  }
            \charf{ M \textrm{ odd or } \lambda_2  \! -  \! \lambda_3 \textrm{ odd}  }\Big. .
    \end{eqnarray*}
The inequalities inside the large brackets above may be made strict as $k$ is odd in this case, and $M$ is required to be even in the first term, and $\lambda_2  \! -  \! \lambda_3$ is required to be even in the second term.

  The result then follows from $\kroncoeff_{\mu\nu\lambda}= (\Sigma - \alpha) + K_{\textrm{even}}+ (K_{\textrm{odd}}+\alpha)$.\end{proof}
}

\begin{section}{Combinatorial and symmetric function consequences}\label{sec:combinatorics}
\begin{subsection}{Tableaux of height less than or equal to $4$}
Since every partition of length less than or equal to $4$ lies in
either $P$ or $\overline{P}$, Corollary \ref{cor:easycases} has as
a consequence the following corollary.
\begin{corollary}\label{cor:boundedsum4} For $d$ a positive integer,
\begin{equation}
\sum_{\lambda \vdash 2d,\ell(\lambda) \leq 4} s_\lambda=
s_{(d,d)} \ast (s_{(d,d)} + s_{(d+1,d-1)})~.
\end{equation}
\begin{equation}
\sum_{\lambda\vdash 2d-1, \ell(\lambda)\leq 4} s_\lambda=
s_{(d,d-1)} \ast s_{(d,d-1)}~.
\end{equation}
\end{corollary}

\begin{proof}
For the sum over partitions of $2d$,  \eqref{sddsdd} (or  \eqref{cleaneq}) 
says that $s_{(d,d)} \ast s_{(d,d)}$ is the
sum over all $s_\lambda$ with $\lambda \vdash 2d$ having four even parts or four odd
parts and Corollary \ref{cor:easycases} says that $s_{(d,d)} \ast s_{(d+1,d-1)}$ is the sum over
$s_{\lambda}$ with $\lambda \vdash 2d$ where $\lambda$ does not have 
four odd parts or four even parts.  Hence, $s_{(d,d)} \ast s_{(d,d)} + s_{(d,d)} \ast s_{(d+1,d-1)}$
is the sum over $s_\lambda$ where $\lambda$ runs over all partitions with less than or equal
to $4$ parts.

For the other identity, we use  \eqref{multperpkron} to derive
$$\left< s_{(d,d-1)} \ast s_{(d,d-1)}, s_\lambda \right>
= \left< s_{(d,d)} \ast s_{(d,d)}, s_{(1)} s_\lambda \right>~.$$
If $\lambda$ is a partition of $2d-1$, then the expression is $0$ if
$\ell(\lambda)>4$ and if $\ell(\lambda) \leq4$ then
$s_{(1)} s_\lambda$ is a sum of at most $5$ terms, $s_{(\lambda_1+1, \lambda_2, \lambda_3, \lambda_4)}, s_{(\lambda_1, \lambda_2+1, \lambda_3, \lambda_4)}, 
s_{(\lambda_1, \lambda_2, \lambda_3+1, \lambda_4)}, 
s_{(\lambda_1, \lambda_2, \lambda_3, \lambda_4+1)}, s_{(\lambda_1, \lambda_2, \lambda_3,
\lambda_4,1)}$.
Because $\lambda$ has exactly $3$ or $1$ terms that are odd,
exactly one of these will have 4 even parts or 4 odd parts.
\end{proof}

The expressions for sums of Schur functions indexed by partitions
of length less than or equal to
$2$ and less than or equal to $3$ can be expressed using the outer
product on Schur functions.  It follows from the Pieri rule that
\begin{equation} \label{eq:boundedsum2}
\sum_{\lambda \vdash n, \ell(\lambda) \leq 2} s_\lambda = 
s_{(\lfloor n/2 \rfloor)} s_{(\lfloor (n+1)/2 \rfloor)}
\end{equation}
and
\begin{equation}\label{eq:boundedsum3}
\sum_{\lambda \vdash n, \ell(\lambda) \leq 3} s_\lambda = 
\sum_{d =0}^{\lfloor n/2 \rfloor} s_{(d,d)} s_{(n-2d)}~.
\end{equation}
For sums of Schur functions with length less than or equal to $5$ this
can also be represented using the outer product, though less trivially, as
\begin{equation}\label{eq:boundedsum5}
\sum_{\lambda \vdash n, \ell(\lambda) \leq 5} s_\lambda = 
\sum_{d =0}^{\lfloor n/2 \rfloor} \sum_{r=0}^{(n-2d)/4} s_{(d+r,d+r,r,r)} s_{(n-2d-4r)}~.
\end{equation}
This formula will generalize to an $\ell$-fold sum for
the sum of partitions less than or equal to $2\ell+1$.

Regev \cite{Regev}, Gouyou-Beauchamps \cite{Gouyou-Beauchamps}, 
Gessel \cite{Gessel} and subsequently Bergeron, Krob, 
Favreau, and Gascon \cite{BFK, BG} studied tableaux of bounded height.
Gessel \cite{Gessel} remarks that all of the expressions for the number
of standard tableaux of height less than or equal to $k$ for $k= 2,3,4,5$ have
an expression simpler than the $k$-fold sum that one would expect to see.  The
first three of those follow immediately from the expressions above.  If we set
$y_k(n) = $ the number of standard tableaux of height less than or equal to $k$
then 
\begin{equation}\label{eq:two}y_2(n) = {{n} \choose {\lfloor n/2 \rfloor}}~,
\end{equation}
\begin{equation}\label{eq:three}
y_3(n) = \sum_{d=0}^{\lfloor n/2 \rfloor} {{n} \choose{2d} }C_d~.
\end{equation}
We can also derive from these results here the previously known
corollary.
\begin{corollary}
\begin{equation}\label{eq:four}
y_4(n) = C_{\lfloor (n+1)/2 \rfloor} C_{\lceil (n+1)/2 \rceil}~.
\end{equation}
\end{corollary}
It follows from the hook length formula that the
number of standard tableaux of shape $(d,d)$ is $C_d$, the number of
standard tableaux of $(d,d-1)$ is $C_d$, and the
number of standard tableaux of shape that are either of shape
$(d,d)$ or $(d+1,d-1)$ is $C_{d+1}$.
Note that \eqref{eq:two}, \eqref{eq:three} and \eqref{eq:four} follow respectively
from \eqref{eq:boundedsum2} and \eqref{eq:boundedsum3} and Corollary
\ref{cor:boundedsum4} and the algebra homomorphism that sends
the Schur function $s_\lambda$ to the number of standard tableaux of shape
$\lambda$ divided by $n!$.  

The formula for the tableaux of height less than or
equal to $5$ \cite{Gouyou-Beauchamps},
$$y_5(n) = \sum_{d=0}^{\lfloor n/2 \rfloor} {{n} \choose {2d}} 
\frac{6 (2d+2)! C_d}{(d+2)! (d+3)!}~,$$
can be derived from \eqref{eq:boundedsum5} although some
non-trivial manipulation is required to arrive at this expression.

\end{subsection}

\begin{subsection}{Generating functions for partitions with coefficient $r$ in $s_{(d,d)} \ast
s_{(d+k, d-k)}$}

An easy consequence of Theorem \ref{cleanest} is a generating function formula for
the the sum of the coefficients of the expressions $s_{(d,d)} \ast s_{(d+k,d-k)}$. \comment{ Using
Corollary \ref{cor:boundedsum4},}
\begin{corollary}  \label{sumofcoeffs} For a fixed $k \geq 1$,
\begin{align*}
G_k(q) &:= \sum_{d \geq k} \left( \sum_{\lambda \vdash 2d} \langle s_{(d,d)} \ast s_{(d+k,d-k)},
s_\lambda \rangle \right) q^d \\
&= \frac{ q^k + q^{k+1} + q^{2k+1} + \sum_{r=k+2}^{2k} 2 q^r}{(1-q)(1-q^2)^2(1-q^3)}~.
\end{align*}
\end{corollary}
\begin{remark}
Note that Corollary \ref{sumofcoeffs} only holds for $k>0$.
In the case that $k=0$ in the formula in the previous expression,
the numerator of the above expression is a bit different and we have from
Corollary \ref{cor:easycases},
\begin{align*}
G_0(q) = \sum_{d \geq 0} \left( \sum_{\lambda \vdash 2d} \langle s_{(d,d)} \ast s_{(d,d)},
s_\lambda \rangle \right) q^d 
&=
\sum_{d \geq 0} \langle s_{(d,d)} \ast s_{(d,d)},
s_{(d,d)} \ast s_{(d,d)} \rangle  q^d\\
&=
\sum_{d \geq 0} \langle \rug{P},
\rug{P} \rangle  q^d = \sum_{d\geq 0} |P| q^d\\
&= \frac{ 1}{(1-q)(1-q^2)^2(1-q^3)}~.
\end{align*}
This last equality is the formula given in \cite[Corollary 1.2]{GWXZ2} and it follows because
the generating function for partitions with even parts and length less than or equal to $4$
is $\frac{ 1}{(1-q)(1-q^2)(1-q^3)(1-q^4)}$, and the generating function for the partitions of $2d$
with odd parts of length less than or equal to $4$ is $\frac{q^2}{(1-q)(1-q^2)(1-q^3)(1-q^4)}$.
The sum of these two generating functions will be equal to a generating function
for the number of non-zero coefficients of $s_{(d,d)} \ast s_{(d,d)}$.
\end{remark}
\begin{proof}
Recall that Theorem \ref{cleanest} gives a formula for the expression of
$s_{(d,d)} \ast s_{(d+k,d-k)}$ in terms of rugs of the form $\rug{ \gamma P}$.  
We can calculate that for
each of the rugs that appears in this expression
\begin{align*}
\sum_{d \geq k} \left( \sum_{\lambda \vdash 2d} \langle \rug{  \gamma P},
s_\lambda \rangle \right) q^d &= 
\sum_{d \geq k} \left( \sum_{\lambda \vdash 2d} \langle \rug{ P },
s_{\lambda - \gamma} \rangle \right) q^d\\
&= \sum_{d \geq k} \left( \sum_{\mu \vdash (2d-|\gamma|)} \langle \rug{ P },
s_{\mu} \rangle \right) q^{d+|\gamma|/2}\\
&= \frac{q^{|\gamma|/2}}{(1-q)(1-q^2)^2(1-q^3)}~.
\end{align*}

Now since $s_{(d,d)} \ast s_{(d+k,d-k)}$ is the sum of rugs of the form
$\rug{ \gamma P}$ where each of the $\gamma = (k,k), (k+1,k,1)$
and $(2k+1,k+1,k)$ each contribute a term to the numerator of the form
$q^k, q^{k+1}$ and $q^{2k+1}$ respectively. 
The rugs with $\gamma$ equal to
$(k+i+1,k+1,i)$ and $(k+i+1,k,i+1)$ for $1 \leq i \leq k-1$
each contribute a term $2 q^{k+i+1}$ to the numerator.
\end{proof}

To allow us to compute other generating functions of
Kronecker products we require the following very surprising
theorem.  It says that the partitions such that $\kroncoeff_{(d,d)(d+k,d-k)\lambda}$
are of coefficient $r>1$ are exactly the partitions $\gamma+(6,4,2)$ where
$\kroncoeff_{(d-6,d-6)(d-6+(k-2),d-6-(k-2))\gamma}$ is equal to $r-1$.
\begin{theorem} \label{magicpartition}
For $k \geq 2$, assume that $\kroncoeff_{(d,d)(d+k,d-k)\lambda}>0$, then
$$\kroncoeff_{(d+6,d+6)(d+k+8,d-k+4)(\lambda + (6,4,2))} 
= \kroncoeff_{(d,d)(d+k,d-k)\lambda}+1~.$$
\end{theorem}

We require the following lemma.
\begin{lemma} \label{hitlemma} For $\gamma$ a partition with $\ell(\gamma) \leq 4$,
$$\lambda \in \gamma P \Leftrightarrow
\lambda + (6,4,2) \hbox{ is in both } (\gamma_1 + 2, \gamma_2+2, \gamma_3, \gamma_4) P\hbox{ and }(\gamma_1 + 4, \gamma_2+2, \gamma_3+2, \gamma_4) P~.$$
\end{lemma}

\begin{proof} $(\Rightarrow)$
If $\lambda \in \gamma P$, then $\lambda - \gamma$ is a partition with four even parts or $4$ odd parts.  Hence, both $\lambda - \gamma + (2,2) = (\lambda + (6,4,2)) - (\gamma + (4,2,2))$ and
$\lambda - \gamma + (4,2,2) = (\lambda + (6,4,2)) - (\gamma + (2,2))$ are elements of $P$.

$(\Leftarrow)$ Assume that $\lambda + (6,4,2)$ is an element of both $(\gamma_1 + 2, \gamma_2+2, \gamma_3, \gamma_4) P$ and $(\gamma_1 + 4, \gamma_2+2, \gamma_3+2, \gamma_4) P$, then
$\lambda_1+6 - (\gamma_1+4) \geq \lambda_2+4 - (\gamma_2+2)$, $\lambda_2+4-(\gamma_2+2)\geq
\lambda_3+2-\gamma_3$, and  $\lambda_3+2-(\gamma_3+2) \geq \lambda_4-\gamma_4 \geq 0$.  This implies that $\lambda - \gamma$ is a partition and since $\lambda-\gamma+(2,2)$ has four even or four odd parts, then so does $\lambda - \gamma$ and hence $\lambda  \in \gamma P$.
\end{proof}

\begin{proof} (of Theorem \ref{magicpartition})  Consider the case where $\lambda$ is a partition of
$2d$ with
$\lambda_2 - k \equiv \lambda_4~(mod~2)$ since the case where
$\lambda_2 - k \equiv\!\!\!\!\!/~~~\lambda_4~(mod~2)$ is analogous and just uses a different
non-zero terms in the sum below.
From Theorem \ref{cleanest}, we have 
\begin{equation}\label{redsum}
\kroncoeff_{(d,d)(d+k,d-k)\lambda} = \sum_{i=0}^k \charf{\lambda \in (k+i,k,i)P}~,
\end{equation}
since the other terms are clearly zero in this case.  If $\lambda_2 - k \geq \lambda_3$, then
the terms in this sum will be non-zero as long as
$0 \leq i \leq \lambda_3 - \lambda_4$ and $0 \leq k+i \leq \lambda_1-\lambda_2$ and $\lambda_3 - i \equiv \lambda_4~(mod~2)$.  Consider the case where
$\lambda_3 \equiv \lambda_4~(mod~2)$\comment{(the other case is similar)}, then
\eqref{redsum} is equal to 
$a+1$ where $a=\lfloor min(\lambda_3 - \lambda_4, \lambda_1 - \lambda_2 - k, k)/2 \rfloor$ since
the terms that are non-zero in this sum are $\charf{\lambda \in (k+2j, k, 2j)P}$ where $0 \leq j \leq a$.
By Lemma \ref{hitlemma}, these
terms are true if and only if $\charf{ \lambda+(6,4,2) \in (k+2+2j,k+2,2j) P}$ 
are true for all $0 \leq j \leq a+1$.  But again by Theorem \ref{cleanest}, in this case we also have
$$\kroncoeff_{(d+6,d+6)(d+k+8,d-k-4)(\lambda+(6,4,2))} = 
a+2 = \kroncoeff_{(d,d)(d+k,d-k)\lambda} + 1~.$$
The case where $\lambda_3 \equiv \lambda_4 + 1~(mod~2)$ is similar,
but the terms of the form $\charf{\lambda\in (k+2j+1,k,2j+1)P}$ in \eqref{redsum}
are non-zero if and only if the terms $\charf{\lambda+(6,4,2)\in (k+2+2j+1,k+2,2j+1)P}$
contribute to the expression for $\kroncoeff_{(d+6,d+6)(d+k+8,d-k-4)(\lambda+(6,4,2))}$
and there is exactly one more non-zero term hence $\kroncoeff_{(d+6,d+6)(d+k+8,d-k-4)(\lambda+(6,4,2))} = \kroncoeff_{(d,d)(d+k,d-k)\lambda} + 1$.
\comment{Now consider $\lambda_3 \equiv \lambda_4 + 1~(mod~2)$, then
\eqref{redsum} is equal to $a$ or $a+1$ depending if $min(\lambda_3 - \lambda_4, \lambda_1 - \lambda_2 - k, k) \equiv 0~(mod~2)$ because the
terms $\charf{ \lambda\in (k+2j+1,k,2j+1)P}$ for 
$0 \leq 2j+1 \leq min(\lambda_3 - \lambda_4, \lambda_1 - \lambda_2 - k, k)$.  Just as before,
Lemma \ref{hitlemma} implies that the terms $\charf{ \lambda+(6,4,2) \in (k+2+2j+1,k+2,2j+1)P}$
are non-zero for $0 \leq 2j+1 \leq min(\lambda_3 - \lambda_4, \lambda_1 - \lambda_2 - k, k)+2$.}
\end{proof}

Now for computational purposes it is useful to have a way of determining 
exactly the number of partitions of $2d$ that have a given coefficient.  
For integers $d, k, r > 0$, we let $L_{d,k,r}$ 
be the number of partitions $\lambda$ of $2d$ with $\langle s_{(d,d)} \ast
s_{(d+k,d-k)}, s_\lambda \rangle = r$.  Our previous theorems have shown that
$L_{d,k,r} = 0$ for $r > \lfloor k/2 \rfloor + 1$ (see Theorem \ref{cleanest})
and $L_{d,k,r} = L_{d-6,k-2, r-1} + 1$ for $r>1$ (see Theorem \ref{magicpartition}).
These recurrences will allow us to completely determine the generating functions
for the coefficients $L_{d,k,r}$.  For this purpose, we set
$$L_{k,r}(q) = \sum_{d \geq 0} L_{d,k,r} q^d~.$$

\begin{corollary}
With the convention that $G_k(q) = 0$ for $k<0$, then
$L_{k,r}(q) = 0$ for $r>\lfloor k/2 \rfloor + 1$, and
\begin{equation} \label{eq:exactcoeff}
L_{k,1}(q) = G_{k}(q) - 2q^6 G_{k-2}(q) + q^{12} G_{k-4}(q)
\end{equation}
\begin{equation}\label{eq:magicpartition}
L_{k,r}(q) = q^{6r-6} L_{k-2r+2,1}(q)~.
\end{equation}
\end{corollary}

\begin{proof}
Theorem \ref{magicpartition} explains  \eqref{eq:magicpartition}
because
\begin{align*}
L_{k,r}(q) &= \sum_{d \geq 0} \#\{ \lambda : C_{(d,d)(d+k,d-k)\lambda} = r \} q^d\\
&= \sum_{d \geq 0} \#\{ \lambda : C_{(d-6,d-6)(d+k-8,d-k+4)(\lambda+(6,4,2))} = r-1 \} q^d\\
&= \sum_{d \geq 0} \#\{ \lambda : C_{(d-6r+6,d-6r+6)(d+k-8r+8,d-k+4r-4)(\lambda+(6r-6,4r-4,2r-2))} 
= 1 \} q^d\\
&= q^{6r-6} \sum_{d \geq 0} \#\{ \lambda : C_{(d-6r+6,d-6r+6)(d+k-8r+8,d-k+4r-4)(\lambda
+(6r-6,4r-4,2r-2))} 
= 1 \} q^{d-6r+6}\\
&= q^{6r-6} L_{k-2r+2,1}(q)~.
\end{align*}

Now we also have by definition and Theorem \ref{cleanest} that
\begin{equation}\label{eq:recgf}
G_k(q) = \sum_{r = 1}^{\lfloor k/2 \rfloor + 1} r L_{k,r}(q)~.
\end{equation}
Hence, we can use this formula and  \eqref{eq:magicpartition} 
to define $L_{k,r}(q)$ recursively.  It remains to show that the formula for $L_{k,1}(q)$
stated in \eqref{eq:exactcoeff} satisfies this formula, which we do by induction.
Given that $L_{0,1}(q) = G_0(q)$ and $L_{1,1}(q) = G_{1}(q)$ and $L_{k,1}(q)=0$ for $k<0$, 
then assuming that the formula holds for values smaller than $k>1$ then 
\eqref{eq:recgf} yields
\begin{align*}
L_{k,1}(q) &= G_{k}(q) - \sum_{r=2}^{\lfloor k/2 \rfloor + 1} r L_{k,r}(q)\\
&= G_{k}(q) - \sum_{r\geq2} r q^{6r-6}L_{k-2r+2,1}(q)\\
&= G_{k}(q) - \sum_{r\geq2} r q^{6r-6}(G_{k-2r+2}(q) - 2 q^6 G_{k-2r}(q) 
+ q^{12} G_{k-2r-2}(q))\\
&= G_{k}(q) 
- \sum_{r\geq1} (r+1) q^{6r}G_{k-2r}(q) 
+ \sum_{r\geq2} 2 r q^{6r} G_{k-2r}(q) 
- \sum_{r\geq3} (r-1) q^{6r} G_{k-2r}(q)\\
&= G_{k}(q) - 2 q^{6} G_{k-2}(q) + q^{12}G_{k-4}(q)~.
\end{align*}
Therefore, by induction we have that \eqref{eq:exactcoeff} holds for all $k>0$.
\end{proof}

\end{subsection}

\end{section}

\end{document}